\documentclass[11pt]{article}
\usepackage{amsmath,amssymb,amsthm}

\newcommand{\res}{\operatorname{Res}}
\newcommand{\coker}{\operatorname{coker}}
\newcommand{\spn}{\operatorname{Span}}
\newcommand{\Laur}{\operatorname{Laurent}}
\newcommand{\LP}{\operatorname{\mathbf{LP}}}
\newcommand{\modd}{\operatorname{Mod}}
\newcommand{\markk}{\operatorname{Mark}}
\newcommand{\bff}{\mathbf f}
\newcommand{\bg}{\mathbf g}
\newcommand{\bn}{\mathbf n}
\newcommand{\G}{_{\Gamma}}
\newcommand{\caB}{\mathcal B}
\newcommand{\caC}{\mathcal C}
\newcommand{\caD}{\mathcal D}
\newcommand{\caE}{\mathcal E}
\newcommand{\caF}{\mathcal F}
\newcommand{\caI}{\mathcal I}
\newcommand{\caL}{\mathcal L}
\newcommand{\caO}{\mathcal O}
\newcommand{\caP}{\mathcal P}
\newcommand{\caR}{\mathcal R}
\newcommand{\caS}{\mathcal S}
\newcommand{\caU}{\mathcal U}
\newcommand{\caV}{\mathcal V}
\newcommand{\bP}{\mathbf P}
\newcommand{\subE}{_{\mathcal E}}
\newcommand{\subP}{_{\mathbb P^1}}
\newcommand{\subdz}{_{dz^2}}
\newcommand{\subtau}{_{\tau}}

\newcommand{\omCB}{\omega_{\caC/\caB}}
\newcommand{\omVD}{\omega_{V/D}}
\newcommand{\omVS}{\omega_{V\times S/D\times S}}
\newcommand{\omRS}{\omega_{\caR/S\times D^m}}
\newcommand{\piVD}{\pi:V\rightarrow D}
\newcommand{\piCB}{\Pi:\caC\rightarrow\caB}
\newcommand{\piRS}{\Pi:\caR\rightarrow S\times D^m}

\newcommand{\dzdw}{\big(\frac{dz}{z}-\frac{dw}{w}\big)}
\newcommand{\Lsubuv}{\caL_{\{u,v\}}}

\newtheorem{thrm}{Theorem}
\newtheorem{lemma}[thrm]{Lemma}
\newtheorem{corollary}[thrm]{Corollary}
\newtheorem{defn}[thrm]{Definition}
\newtheorem*{thrm*}{Theorem}

\usepackage{graphicx}
\graphicspath{%
    {converted_graphics/}
    {CBMS/CBMSbook/}
}
\begin{document}

\title{Infinitesimal deformations of nodal stable curves\footnote{2010 {\em Mathematics Subject Classification.} Primary 14H15, 32G15.}}         
\author{Scott A. Wolpert\footnote{Partially supported by National Science Foundation grant DMS - 1005852.}}        
\date{\today}          

\maketitle
\begin{abstract}
An analytic approach and description are presented for the moduli cotangent sheaf for suitable stable curve families including noded fibers.  For sections of the square of the relative dualizing sheaf, the residue map at a node gives rise to an exact sequence.  The residue kernel defines the vanishing residue subsheaf.  For suitable stable curve families, the direct image sheaf on the base is locally free and the sequence of direct images is exact.  Recent work of Hubbard-Koch and a formal argument provide that the direct image sheaf is naturally identified with the moduli cotangent sheaf.  The result generalizes the role of holomorphic quadratic differentials as cotangents for smooth curve families.  Formulas are developed for the pairing of an infinitesimal opening of a node and a section of the direct image sheaf.  Applications include an analytic description of the conormal sheaf for the locus of noded stable curves and a formula comparing infinitesimal openings of a node.  The moduli action of the automorphism group of a stable curve is described.  An example of plumbing an Abelian differential and the corresponding period variation is presented.   
\end{abstract}

\section{Introduction.}

A torus $T$ with fundamental group marking is uniformized by the complex plane $\mathbb C$ with variable $z$ and a lattice generated by $1$ and $\tau$,  $\tau$ in the upper half plane $\mathbb H$.  The change of marking equivalence relation for tori is given by the action of the modular group $SL(2;\mathbb Z)$ on $\mathbb H$.  The Gr\"{o}tzsch and Rauch variational formulas provide that the differential of the moduli parameter $\tau$ is represented by the quadratic differential $-2idz^2\in H^0(T_{\tau},\caO(K^2))$.  With the analogy to higher genus moduli in mind, consider $\mathbb H$ as the Teichm\"{u}ller space, $SL(2;\mathbb Z)$ as the mapping class group and the quotient as the moduli space.   The compactification of the quotient $\mathbb H/SL(2;\mathbb Z)$ is given by introducing the coordinate $\mathbf t=e^{2\pi i\tau}$ for a neighborhood of infinity.  The differential of the nonzero moduli parameter $\mathbf t$ is $4\pi \mathbf tdz^2\in H^0(T_{\tau},\caO(K^2))$, which formally vanishes at infinity or equivalently $d\mathbf t/\mathbf t$ is represented by the quadratic differential $4\pi dz^2$.  In particular, quadratic differentials model the logarithmic derivative of the moduli parameter at infinity.  Equivalently, the moduli cotangent $4\pi\mathbf tdz^2$ at infinity includes a $\mathbf t$ factor.

The infinitesimal deformation space of a pair $(R,q)$, a compact Riemann surface and a distinguished point, is the cohomology group $H^1(R,\caO(K^{-1}q^{-1}))$ for the canonical bundle $K$ and the inverse of the point bundle.  By Kodaira-Serre duality, the dual infinitesimal deformation space is $H^0(R,\caO(K^2q))$, the space of holomorphic quadratic differentials with possible simple poles at $q$.  Quadratic differentials with simple poles give the moduli cotangent space.  Our goal is to generalize this result for nodal stable curves and describe the moduli cotangent space as a coherent analytic sheaf on the family.

Our discussion begins with the local geometry of the model case, the family of hyperbolas $zw=t$ in $\mathbb C^2$.  Consider for $c,c'$ positive, the singular fibration of $V=\{|z|<c,|w|<c'\}\subset\mathbb C^2$ over $D=\{|t|<cc'\}$ with the projection $\pi(z,w)=t$.  The family $\piVD$ is the local model for the formation and deformation of a node.  The general fiber is an annulus and the special fiber is the union of the germ of the coordinate axes with the origin the node.  The projection differential is $d\pi=zdw+wdz$.  The vertical line bundle $\caL$ over $V$ has non vanishing section 
$z\partial/\partial z-w\partial/\partial w$ and the relative dualizing sheaf $\omVD$ has non vanishing section $dz/z-dw/w$.     

We consider curve families with noded stable fibers. A proper surjective map $\piCB$ of analytic spaces is a family of nodal curves provided at each point $\mathbf c\in\caC$, either $\Pi$ is smooth with one dimensional fibers, or the family is locally analytically equivalent to a locus $zw=f(s)$ in $\mathbb C^2\times\caS$ over $\caS$ with $f$ vanishing at $\Pi(\mathbf c)$.  The locus $\{(0,0)\}\times\caS$ is the loci of nodes.  We consider families 
 with first order vanishing of $f$; for a neighborhood of a node the family is analytically equivalent to a Cartesian product of $\piVD$ and a complex manifold.  Stability of fibers is the condition that each component of the nodal complement in a fiber has negative Euler characteristic.  Negative Euler characteristic ensures that the automorphism group of a component is finite. The general fiber of a family $\piCB$ of nodal curves is a smooth Riemann surface.  The locus of noded curves within the family is a divisor with normal crossings.  For a family $\piCB$, the relative dualizing sheaf is isomorphic to the product of canonical bundles $K_{\caC}\otimes \Pi^*K_{\caB}^{\vee}$.

\begin{figure}[htbp] 
  \centering
  \includegraphics[bb=0 0 556 556,width=2.4in,height=2.4in]{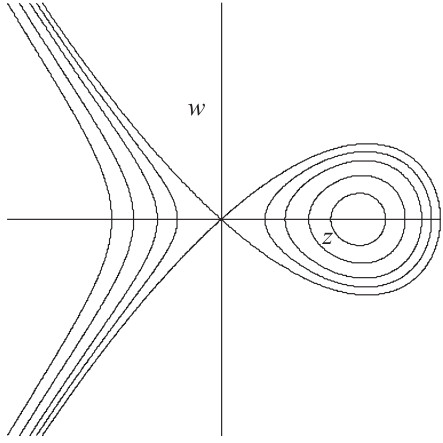}
  \caption{A genus $1$ degenerating family $w^2=(1-z-t)(z^2-t)$.}
  \label{fig:degngen1}
\end{figure}

The relative dualizing sheaf $\omCB$ provides a generalization of the family of canonical bundles for a family of Riemann surfaces. For $k$ positive, sections of $\omCB^k$ over open sets of $\caB$, generalize families of holomorphic $k$-differentials for Riemann surfaces.  A section $\eta$ of $\omCB^k$ on a neighborhood of a node $\pi:V\times S\rightarrow D\times S$  is given as
\[
\eta\,=\,\mathbf f(z,w,s)\dzdw^k
\]
with $\mathbf f$ holomorphic in $(z,w,s)$.  For an annulus fiber and the mapping $z=\zeta,w=t/\zeta,t\ne 0,$ into a fiber, the section is given as
\begin{equation}\label{Laurcoeff}
\eta\,=\,\mathbf f(\zeta,t/\zeta,s)\big(2\frac{d\zeta}{\zeta}\big)^k.
\end{equation}
Sections of $\omCB^2$ generalize families of holomorphic quadratic differentials and in the language of Bers are families of regular $2$-differentials \cite{Bersdeg}.  Infinitesimal opening of a node is described by the variational formula for the parameter $t$.  For $t$ nonzero, the pairing of the infinitesimal variation of $t$ with the section $\eta$ is given in Lemma \ref{plumderiv} as the $-\pi/t$ multiple of the $\zeta$-constant coefficient in the Laurent expansion of $4\mathbf f(\zeta,t/\zeta,s)$.   The appearance of the scaling $t$-factor is intrinsic to the variation of a node.  The present appearance of the $t$-factor is dual to the appearance in the torus example.   

The noded fibers sub family of $\pi:V\times S\rightarrow D\times S$ is $\{(z,w)\mid z=0 \mbox{ or } w=0\}\times S$ over $S$ with the restriction of the section $\eta$ to the fibers given as $\mathbf f(z,0,s)(dz/z)^k$ on $\{w=0\}$ and as $\mathbf f(0,w,s)(-dw/w)^k$ on $\{z=0\}$.  The residue of $\eta$ at the node is $\pm\mathbf f(0,0,s)$.  The residue is coordinate independent and well defined modulo a sign; for $k$ even the residue is well defined.  The kernel of the residue map to $\mathbb C$ defines a subsheaf of $\omCB^2$ as follows.  For a family $\piCB$, let $\mathbf n_k, 1\le k\le m,$ be the $k^{th}$ component of the loci of nodes in $\caC$ and let $\mathbb C_ {\mathbf n_k}$ be the skyscraper sheaf supported on $\mathbf n_k$.  The vanishing residue subsheaf is defined by the exact sequence

\begin{equation*}
0\longrightarrow\caV\longrightarrow\omCB^2\stackrel{\oplus_k\res_k}{\longrightarrow}\oplus_k\mathbb C_{\bn_k}\longrightarrow 0.
\end{equation*}
In the preparatory Lemma \ref{frameeval}, we describe local bases of sections of $\omCB^2$ over suitable opens sets of $\caB$.  The bases have a direct relationship to coordinate cotangent frames for the open sub family of Riemann surfaces.   Lemma \ref{frameeval} provides the basic tool for understanding the vanishing residue subsheaf and the moduli cotangent sheaf.  In Section 4, we show that the direct image $\Pi_*\caV$, with presheaf of sections of $\caV$ over open sets of $\caB$, is locally free and that the corresponding sequence of direct images is also exact.  We then combine the coordinate cotangent frame for sub families property, the definition of $\caV$ and a formal argument to establish the main result.

\begin{thrm*}
The vanishing residue exact sequence is natural for admissible families.  The direct image of the vanishing residue subsheaf is naturally identified with the cotangent sheaf for admissible families.  
\end{thrm*}

We discuss applications.  A section of the moduli cotangent sheaf over an open set of the base is a section of $\omCB^2$ over the open set with residues vanishing.  In the torus example the section vanishes along the fiber and necessarily the residue vanishes.  A description of the moduli cotangent fiber at a nodal curve is provided by the isomorphism between a locally free sheaf and an analytic vector bundle.  For deformations of a Riemann surface, moduli cotangents are represented by quadratic differentials on the surface.  For deformations of a nodal stable curve, moduli cotangents are represented by equivalence classes of sections of the vanishing residue subsheaf - it is not sufficient to specify the restriction of a section to a fiber of a family.   In Lemma \ref{cotfib}, we find that the moduli tangent-cotangent pairing extends to noded fibers.  We also consider the basic geometry of the divisor of noded curves within the family.   We find that the conormal sheaf of the divisor is the subsheaf of $\omCB^2$ with sections over base open sets vanishing along all noded fibers.  The log-cotangent sheaf of the divisor is the direct image of $\omCB^2$, with sections of $\omCB^2$ over base open sets.

The works of Masur \cite{Msext} and Hubbard-Koch \cite{HKmg} provide a foundation for the present considerations. In Section 3, we give the construction of a standard nodal stable family of curves following Masur.  The construction begins with an open set in Teichm\"{u}ller space parameterizing Riemann surfaces with distinguished points and combines plumbing copies of the family $\piVD$.  The constructed family provides the setting for our considerations.  Hubbard and Koch introduce the notion of a $\Gamma$-marking for families of nodal stable curves, a marking modulo Dehn twists about the elements of a multi curve.  They construct a universal analytic family for a $\Gamma$-marking and show that the Deligne-Mumford compactification of the moduli space of Riemann surfaces is analytically described by quotients of $\Gamma$-marked universal families.   From the results of Hubbard and Koch, our considerations apply to the Deligne-Mumford compactification.  

In Section 5, we generalize the Laurent coefficient map of Lemma 4 for expansion (\ref{Laurcoeff}).  The generalization is a period about a collar latitude of the product of a section of the vertical line bundle and a section of $\omCB^2$.  The generalization enables  direct comparison of the infinitesimal plumbings $F(z)G(w)=\tau$ and $zw=t$ of a node.    

\begin{thrm*}
Let $\piCB$ be a proper family of stable curves with a nodal fiber $\caC_b$ with a neighborhood of a node analytically equivalent to the Cartesian product of $\piVD$ and a parameter space $S$.  Consider the local coordinates $F(z)$ and $G(w)$, $F(0)=G(0)=0$, for neighborhoods of the inverse images of the node on the normalization of $\caC_b$.  For the plumbing $F(z)G(w)=\tau$  of $\caC_b$ contained in $\piCB$ and $\eta$ a section of $\caV$ on a neighborhood of $\caC_b$, the initial plumbing tangent evaluates as
\begin{multline*}
\big(\frac{\partial}{\partial \tau},\eta\big)\Big|_{\tau=0}\,=\\
\pi(F'(0)G'(0))^{-2}\big(-F'(0)G'(0)\mathbf f_{zw}(0,0,s)\,+\,\frac12F''(0)\mathbf f_w(0,0,s)\,+\,\frac12G''(0)\mathbf f_z(0,0,s) \big).
\end{multline*}
\end{thrm*} 
The formula generalizes the standard fact that for the Deligne-Mumford compactification the normal sheaf to the divisor of noded curves is the direct sum of products of tangent lines at distinguished points on the normalized curves.

In Section 6 we present three applications for the vanishing residue sheaf. The first involves the action of the automorphism group of a stable curve on infinitesimal deformations.  The quotient of the action on the moduli tangent space provides a local model for the Deligne-Mumford compactification.  We apply the main theorem and give a detailed description of the automorphism action on the vanishing residue sheaf.  In the second application, we give a detailed discussion of plumbing an elliptic curve family to a $3$-pointed $\mathbb P^1$.  We describe the initial moduli cotangent space and the action of the elliptic curves and pointed-$\mathbb P^1$ involutions on moduli parameters and moduli cotangents.  In the third application, we apply a standard construction for plumbing an Abelian differential and compare two approaches for calculating the variation of period.               

 I would like to especially thank John Hubbard and Sarah Koch for valuable conversations.  

\section{The local geometry of $zw=t$.} \label{locgeom}    

For a complex manifold, we write $\mathcal O$ for the sheaf of holomorphic functions, $T$ for the holomorphic tangent sheaf, $\Omega$ for the holomorphic cotangent sheaf and $K$ for the canonical bundle, the determinant line bundle of $\Omega$. Equivalently $T$ is the sheaf of holomorphic vector fields and $\Omega$ the sheaf of holomorphic $1$-forms.  We introduce a singular fibration of a neighborhood $V$ of the origin in $\mathbb C^2$ over a neighborhood $D$ of the origin in $\mathbb C$.

\begin{defn}
For positive constants $c,c'$, with $V=\{|z|<c,|w|<c'\}$ and $D=\{|t|<cc'\}$, the family $\piVD$ is the singular fibration with projection $\pi(z,w)=zw=t$. 
\end{defn}

The differential of the projection $d\pi=zdw\,+\,wdz$ vanishes only at the origin; the $t=0$ fiber crosses itself at the origin.  For $t\ne 0$, solving for $z$ the fiber of $\pi$ is $(z,w)$ with $|t|/c'<|z|<c$ and for $t=0$, the fiber is the union of discs $(z,0)$ with $|z|<c$ and $(0,w)$ with $|w|<c'$ in $\mathbb C^2$.  The family $V$ over $D$ is a family of annuli degenerating to a one point union of a $z$ disc and a $w$ disc in the axes of $\mathbb C^2$.  Alternatively, $V$ over $D$ is a representative of the germ at the origin of the family of hyperbolas limiting to the union of coordinate axes. 
\begin{figure}[htbp] 
  \centering
  \includegraphics[bb=0 0 556 556,width=2.5in,height=2.5in,keepaspectratio]{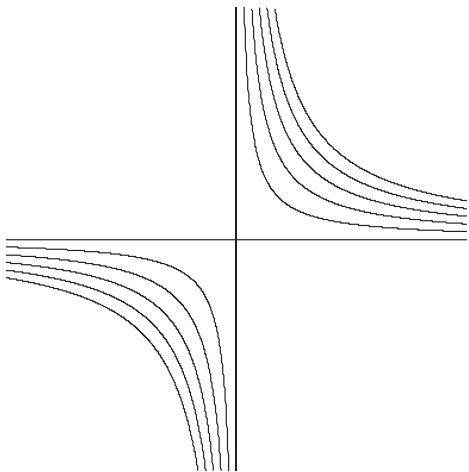}
  \caption{The family of hyperbolas $zw=t$.}
  \label{fig:hyperb}
\end{figure}
The vector field 
\[
v\,=\,z\frac{\partial}{\partial z}-w\frac{\partial}{\partial w}\in T_V,
\]
is vertical on $V-\{0\}$, since $d\pi(v)$ vanishes. Let $v'$ be another vertical vector field, non vanishing on $V-\{0\}$.  Since on $V-\{0\}$, $\ker d\pi$ is rank one, it follows that $v'=fv$, for $f$ a section of $\mathcal O(V-\{0\})$.  By Hartog's Theorem \cite{Narbk}, $f$ is analytic on $V$ and since $v,v'$ are non vanishing on $V-\{0\}$, it follows that $f$ is non vanishing on $V$.  The observations provide that the condition {\em vertical vector field} defines a line bundle $\mathcal L$ over $V$ and $v$ represents a non vanishing section (the vector field $v$, a section of $\Omega_V$, vanishes at the origin; the corresponding line bundle section is non vanishing; $\mathcal L$ is not a sub bundle of $\Omega_V$).  In general, vertical vector fields that vanish at most in codimension $2$ correspond to non vanishing vertical line bundle sections. 

The $0$-fiber of $V$ over $D$ is an example of an open noded Riemann surface \cite[\S 1]{Bersdeg}, alternatively an open nodal curve.  The $0$-fiber is normalized by removing the origin of $\mathbb C^2$ to obtain a $z$ disc, punctured at the origin and a $w$ disc, punctured at the origin. The origins are filled in to obtain disjoint discs.  Analytic quantities on the 
$0$-fiber, lift to analytic quantities on the normalization.  The sheaf of regular $1$-differentials, \cite[\S 1]{Bersdeg}, equivalently the dualizing sheaf \cite{Bard}, \cite[Dualizing sheaves, pg. 82]{HMbook} associates to the $0$-fiber: Abelian differentials $\beta_z,\beta_w$ with at most simple poles respectively at the origin for the $z,w$ discs and the important residue matching condition $\res \beta_z\,+\,\res\beta_w=0$.  
Bardelli provides an exposition on the basics of families of stable curves \cite{Bard}.  He discusses the topology of fibers, the arithmetic genus, the dualizing and relative dualizing sheaves, line bundles and divisors, Riemann Roch, as well as K\"{a}hler differentials and first order deformation theory.

The fiber tangent spaces of $V-\{0\}$ over $D$ are subspaces of $T_V$.  The meromorphic differential
\[
\alpha\,=\,\frac{dz}{z}-\frac{dw}{w},
\]
is a functional on the fiber tangent spaces of $V-\{0\}$. For the $0$-fiber, it is immediate that $\alpha$ is a section of the dualizing sheaf.  The differential $\alpha$ satisfies the relations
\[
\alpha\wedge d\pi\,=\,2dz\wedge dw\qquad\mbox{and}\qquad\alpha(v)\,=\,2.
\]
The differential is uniquely determined modulo the submodule $\mathcal O(d\pi)\subset\Omega_V $  by each relation.  To  motivate the definition of the {\em relative dualizing sheaf} for $V$ over $D$, we consider the coset of $\alpha$ in $\Omega_V/\mathcal O(d\pi)$.  The finite pairing $\alpha(v)$ provides that the coset has holomorphic sections on $V-\{0\}$.  Let $\alpha'$ be another meromorphic differential on $V$ with $\alpha'\wedge d\pi$ holomorphic and non vanishing on $V$.  On $V-\{0\}$, the quotient sheaf $\Omega_V/\mathcal O(d\pi)$ is pointwise rank one and thus $\alpha'=f\alpha$ for $f$ holomorphic on the domain.  Again by Hartog's theorem \cite{Narbk}, the function $f$ is analytic on $V$.  The relation $\alpha'\wedge d\pi=f\alpha\wedge d\pi$ and non vanishing of the first quantity imply that $f$ is non vanishing on $V$.  In particular, the differentials $\alpha'$ with  $\alpha'\wedge d\pi$ non vanishing, considered as elements of $\Omega_V/\mathcal O(d\pi)$, define a line bundle over $V$; the line bundle is not a sub bundle of $\Omega_V$.  The differentials $\alpha'$ correspond to non vanishing holomorphic sections of the line bundle.  The constant relation $\alpha(v)=2$, shows that the line bundle is the dual of the vertical line bundle $\mathcal L$.  The relative dualizing sheaf $\omega_{V/D}$ is defined to have sections of the relative cotangent bundle $\coker(d\pi:\pi^*\Omega_D\rightarrow\Omega_V)$, given by differentials satisfying the polar divisor and residue conditions \cite[Sec. III]{Bard}, \cite[Dualizing sheaves, pg. 84]{HMbook}.  In particular, we have the sheaf equality $d\pi:\pi^*\Omega_D=\mathcal O(d\pi)$ and the quotient $\Omega_V/\mathcal O(d\pi)$ is the intended cokernel.  We have described the relative dualizing sheaf and shown that it is dual to the vertical line bundle $\mathcal L$.  We will use that for the families considered, the restriction of the relative dualizing sheaf to a fiber of the family is the dualizing sheaf of the fiber \cite[Prop. 3.6]{Bard}.  

Since the total space $V$ is smooth, there is a description of $\omega_{V/D}$ in terms of the canonical bundle $K_V$ and pullback $\pi^*K_D$ of the canonical bundle of $D$ \cite[Definition 3.5]{Bard}, \cite[Dualizing sheaves, pg. 84]{HMbook}; in particular
\begin{equation}\label{omegaK}
\omega_{V/D}\,\simeq\, K_V\otimes\pi^*K_D^{\vee},
\end{equation}
where $^{\vee}$ denotes the dual.  As above, a differential $\alpha'$, satisfying the polar divisor and residue conditions, determines a coset in $\Omega_V/\mathcal O(d\pi)$.  A differential $\alpha'$ and a non vanishing section $\beta$ of $\mathcal O(d\pi)$ together determine the element $\alpha'\wedge\beta\otimes\beta^{\vee}$ in $K_V\otimes\pi^*K_D^{\vee}$.  The association $\alpha'\operatorname{mod}\mathcal O(d\pi)\longleftrightarrow \alpha'\wedge\beta\otimes\beta^{\vee}$ is independent of the particular choice of $\beta$ and realizes the sheaf isomorphism (\ref{omegaK}).  

We are interested in general curve families $\piCB$ including nodal curves or equivalently analytic families of noded Riemann surfaces (possibly open) with smooth total space $\caC$ and base $\caB$.  We require that the projection $\Pi$ is a submersion on the complement of a codimension $2$ subset.  Families of open or compact Riemann surfaces and the above nodal family are included in the considerations.  For each node of a fiber, we assume that for a neighborhood the family $\piCB$ is analytically equivalent to a Cartesian product of the standard nodal family $\piVD$ and a complex manifold parameter space $S$.

The vertical line bundle $\caL$ and relative dualizing sheaf $\omCB$ are defined for the family $\piCB$.  The vertical line bundle is defined from $\ker d\Pi$ and the relative dualizing sheaf is $\omCB=\coker(d\Pi:\Pi^*\Omega_{\caB}\rightarrow\Omega_{\caB})$.  The positive powers $\omCB^k$ of the relative dualizing sheaf, alternatively the sheaf of regular $k$-differentials \cite{Bersdeg}, is defined to have sections $\eta$ of $\Omega_{\caC}^k/\Omega_{\caC}^{k-1}\otimes d\Pi(\Omega_{\caC})$ with at most order $k$ poles at the image of the nodes on the normalizations  of nodal fibers and the residue matching $\res\eta_z=(-1)^k\res\eta_w$ for the forms $\eta_z,\eta_w$ on the normalization.  The power $\omCB^k$ is dual to the power $\caL^k$ of the vertical line bundle.  The Cartesian product of $\piVD$ and $S$ provides a local model for the family geometry.  

We present an explicit isomorphism for the sheaf isomorphism
\[
\omCB\,\simeq K_{\caC}\otimes\Pi^*K_{\caB}^{\vee}.
\]
Restricting domains as necessary, let $\beta$ be a non vanishing section of the relative dualizing sheaf $\omCB$ and $\tau$ a non vanishing section of $K_{\caB}$.  Consider the association between sections $\psi$ of $\omCB^k$ and $k$-canonical forms for $\caC$, sections $\Psi$ of $K_{\caC}^k$ - the association is given by the formula
\begin{equation}\label{psiPsi}
\psi\,=\,\frac{\Psi}{(\beta\wedge\Pi^*\tau)^k}\,\beta^k.
\end{equation}  
Observations are in order.  On the submersion set for $\Pi$, a non vanishing section of 
$\omCB=\coker(d\Pi:\Pi^*\Omega_{\mathcal B}\rightarrow\Omega_{\mathcal C})$ and the pullback of a local frame for $\Omega_{\mathcal B}$ together form a frame for $\Omega_{\mathcal C}$.  
It follows that the product $\beta\wedge\Pi^*\tau$ is a non vanishing section of the canonical bundle $K_{\caC}$ on the submersion set.  With the codimension $2$ condition, the product is non vanishing in general - consequently the ratio $\Psi/(\beta\wedge\Pi^*\tau)^k$ is an analytic function on $\caC$.  The relation can be inverted to give a formula for $\Psi$ in terms of $\psi$.  In particular, the association provides a local isomorphism of sheaves. The right hand side depends on the choice of $\tau$, but is homogeneous of degree zero in $\beta$, and so is independent of the particular choice of $\beta$. The association establishes a 
canonical isomorphism between $\omCB^k$ and $K_{\caC}^k$ twisted by $\Pi^*K_{\caB}^k$.  The isomorphism 
\begin{equation}\label{omegaKC}
\omCB^k\simeq (K_{\caC}\otimes \Pi^*K_{\caB}^{\vee})^k 
\end{equation}
is a general form of the isomorphism (\ref{omegaK}).   Most important for our considerations, $k$-canonical forms present a local model for sections of powers of the relative dualizing sheaf, alternatively a model for families of regular $k$-differentials.

We use the local Cartesian product description of the family to define annuli in the fibers and annular Laurent expansions of sections of $\omCB^k$.  The annuli and expansion coefficients depend on the local analytic equivalence of $\piCB$ to a Cartesian product.  First, the fibers of $\piVD$, the annuli $\{|t|/c'<|z|<c\}$ or the one point union of $z$ and $w$ discs, define annuli or a nodal region in the fibers of $\piCB$.  
The differential $dz/z-dw/w$ is a non vanishing section of $\omVD$ and gives a non vanishing section of the relative dualizing sheaf for the Cartesian product. The annulus $A=\{|t|<|\zeta|<1|\}$ maps into the $t$-fiber of $V$ by $\zeta\rightarrow(\zeta,t/\zeta)$ or by $\zeta\rightarrow(t/\zeta,\zeta)$ with the differential $dz/z-dw/w$ pulling back to $\pm 2d\zeta/\zeta$.  The $2$ factor enters in comparing expansions in $(z,w)$ to expansions in $\zeta$.  In particular the differential $dz/z-dw/w$ has residues $\pm1$ when considered on the normalization, while the pulled back differential $\pm2d\zeta/\zeta$ has residue $\pm2$.  For all our considerations we define the residue on the normalization and the annular Laurent coefficients in terms of the pull back to the annulus $A$.  
A section $\eta$ of $\omVS^k$ is given as
\[
\eta\,=\,\mathbf f(z,w,s)\dzdw^k,
\]
since $(z,w,s)$ is a coordinate triple for $V\times S$.  In particular for the coefficient function given by a power series
\begin{equation}\label{fzws}
\mathbf f(z,w,s)\,=\,\sum_{m,n\ge 0}a_{mn}(s)z^mw^n,
\end{equation}
the pullback of the $k$-differential to $A$ for a $t$ nonzero fiber has the annular Laurent series 
\[
\mathbf f(\zeta,t/\zeta,s)\,=\,2^k\sum_{m,n\ge0}a_{mn}(s)\zeta^m(t/\zeta)^n
\]
with $\zeta$-constant term
\begin{equation}\label{zconst}
2^k\sum_{m\ge0}a_{mm}(s)t^m.
\end{equation}
The $0$-fiber of $V$ is the union of $\{w=0\}$ with coordinate $z$ and $\{z=0\}$ with coordinate $w$.  
The $\{w=0\}$ and $\{z=0\}$ restrictions of $\mathbf f$ are the functions
\[
\sum_{m\ge0}a_{m0}(s)z^m\qquad\mbox{and}\qquad\sum_{n\ge0}a_{0n}(s)w^n.
\]
The $\zeta$-constant term is long recognized as important in the analysis and variational theory of embedded annuli \cite{Gardtheta,Hejmono,HKmg,HSS, Kracusp,Msext,McM,HPZur,HPEin,Wlcbms,Wlcurv}.  The residue of the differential $\eta$ at the node of the $t=0$ fiber is simply the zeroth coefficient $a_{00}(s)$. The Laurent $\zeta$-constant term is important in the considerations of the following sections.  We concentrate on the square of the relative dualizing sheaf.  

\begin{defn}\label{defLaur}
For a section $\eta$ of $\omCB^2$ and a local analytic equivalence of $\piCB$ to a Cartesian product $\piVD$ and $S$, define the $\zeta$-constant coefficient 
\[
\Laur(\eta)\,=\,4\sum_{m\ge0}a_{mm}(s)t^m.
\]
\end{defn}\label{Laurdef}
  
We consider $\Laur$ as a sheaf map $\omCB^2\rightarrow\caO_{\caB}$ defined on open sets given as suitable Cartesian products. The sheaf map depends on the Cartesian product equivalence.  The value of $\Laur(\eta)$ at $t=0$ is the fourfold multiple of the residue of $\eta$ on the normalization - the residue is intrinsically defined.  

\section{The $(s,t)$ families and regular quadratic differentials.}

We describe the local deformation space, as well as, the associated family for a nodal stable curve, equivalently for a compact noded Riemann surface. The description includes the family tangent and cotangent coordinate frames.  The discussion closely follows \cite[Sections 2, 5 and 7]{Msext}, \cite[Section 3]{Wlcomp} and especially \cite[Section 2]{Wlhyp}. Similar treatments are found in \cite{LSY1,LSY2} and \cite{HKmg}.    

A Riemann surface with nodes $R$ is a connected complex space, such that every point has a neighborhood analytically isomorphic to either the unit disc in $\mathbb C$ or the germ at the origin of the intersection of the coordinate axes in $\mathbb C^2$.  The special points of $R$ are the nodes.  We write $\tilde R$ for the nodal complement $R-\{nodes\}$.  The normalization of $R$ is $\tilde R$ with the removed nodes considered as distinguished points.  Functions, line bundles and sections of line bundles on $R$ lift to corresponding quantities on $\tilde R$ with a matching condition.  The components of $\tilde R$ are the parts of $R$.  Provided $R$ is compact, each component of $\tilde R$ is described as a compact surface minus a finite number of points.  We now assume that each part has negative Euler characteristic.  In particular a noded Riemann surface is equivalent to a stable curve \cite{Bard,HMbook}.  

In the Kodaira-Spencer setup the infinitesimal deformation space of a compact complex manifold $M$ is the \v{C}ech cohomology $H^1(M,T^{\vee})$ \cite{Kod}. The infinitesimal deformation space of a pair $(R,q)$, a Riemann surface and a distinguished point, is $H^1(R,\caO(K^{-1}q^{-1}))$ for $K$ the canonical bundle and the inverse of the point line bundle $q$ ($\caO(q)$ is the line bundle with a section with divisor $q$).  By the Dolbeault isomorphism $H^1(R,\caO(K^{-1}q^{-1}))\simeq H_{\bar\partial}^{0,1}(R,\caE(K^{-1}q^{-1})),$  $\caE(K^{-1}q^{-1})$ the sheaf of smooth vector fields vanishing at $q$ \cite{Kod}.  The elements of the Dolbeault group are equivalence classes of $(0,1)$ forms with values in smooth vector fields vanishing at $q$ - smooth Beltrami differentials.   A Beltrami differential $\nu$ with support disjoint from $q$ represents a trivial infinitesimal deformation exactly when there is a smooth vector field $F$, vanishing at $q$, with $\bar\partial F=\nu$.   

By Kodaira-Serre duality the dual of the infinitesimal deformations is $H^0(R,\caO(K^2q))$, the space of holomorphic quadratic differentials with a possible simple pole at $q$.  In terms of a local coordinate $z$, the integral pairing for a Beltrami differential $\nu=\nu(z)\partial/\partial z\otimes d\bar z$ and a holomorphic quadratic differential $\phi=\phi(z)dz^2$ is  
\begin{equation}\label{Serrepair}
\int_R\nu(z)\phi(z)\,dE,
\end{equation}
for $dE$ the Euclidean area element in $z$.   Since holomorphic sections vanishing on an open set vanish identically, it follows that given an open set in $R$ there exists a basis for 
$H_{\bar\partial}^{0,1}(R,\caE(K^{-1}q^{-1}))$ of elements supported in the open set.

We describe local coordinates for the Teichm\"{u}ller space of $R$ \cite{Ahsome}.  Associated to a Beltrami differential of absolute value less than unity is a deformation of $R$.  Specifically given an atlas $\{(U_{\alpha},z_{\alpha})\}$ for the surface and a Beltrami differential, define new charts as follows: for $z$ the coordinate on $z_{\alpha}(U_{\alpha})\subset\mathbb C$ and $\nu(z)$ the local expression for $\nu$, let $w_{\alpha}=w^{\nu}(z)$ be a homeomorphism solution of the Beltrami equation $w^{\nu}_{\bar z}=\nu w^{\nu}_z$ on 
$z_{\alpha}(U_{\alpha})$ \cite{AB}. The new atlas $\{(U_{\alpha},w_{\alpha}\circ z_{\alpha})\}$ defines the new surface $R_{\nu}$.  Local holomorphic coordinates for the Teichm\"{u}ller space of $R-\{q_1,\dots,q_k\}$ are given as follows: choose $\nu_1,\dots,\nu_n$ compactly supported, spanning the Dolbeault group $H_{\bar\partial}^{0,1}(R,\caE((Kq_1\cdots q_k)^{-1}))$ for $s=(s_1,\dots,s_n)$ small, $\nu(s)=\sum_j s_j\nu_j$ satisfies $|\nu(s)|<1$ and $R_s=R_{\nu(s)}$ is a Riemann surface.  The assignment $s\rightarrow \{R_s\}$ for $|s|$ small, is a local coordinate for Teichm\"{u}ller space.  We will use that a chart $(U_{\beta},z_{\beta})$ for $R$ disjoint from the supports $\operatorname{supp}(\nu_j)$ is a chart for $R_s$; a fixed set of charts can be used for a neighborhood of the distinguished points.  

The initial variation of a family $R_s$ can be found as follows.  A one-parameter family of solutions of the Beltrami equation $w^{s\nu}_{\bar z}=s\nu w^{s\nu}_z$ has an initial variation $\frac{d}{ds}w^{s\nu}=\dot w[\nu]$  satisfying the $\bar\partial$-potential equation $\dot w[\nu]_{\bar z}=\nu$.  The initial variation $\dot w[\nu]$ is a vector field and for deformations of the pair $(R,q)$, vector fields vanish at $q$.    The data $\{(U_{\alpha},\dot w[\nu]\circ z_{\alpha})\}$ is a $0$-cocycle with values in $\caE(K^{-1}q^{-1})$.  Moreover solutions of the $\bar\partial$-potential equation are unique modulo holomorphic functions.  The \v{C}ech coboundary of $\{(U_{\alpha},\dot w[\nu]\circ z_{\alpha})\}$ is the deformation class in $H^1(R,\caO(K^{-1}q^{-1}))$.    

We are interested in the spaces of regular quadratic differentials for a proper family $\piCB$ of stable curves.  As above, the total spaces $\caC$ and $\caB$ are smooth; $\Pi$ is a submersion on the complement of a codimension $2$ subset; each part of a fiber has negative Euler characteristic.  For a sheaf $\caS$ on $\caC$, the direct image presheaf $\Pi_*\caS$ assigns to open sets $U$ in $\caB$ sections of $\caS$ on $\Pi^{-1}(U)$.  We review that the direct image $\Pi_*\omCB^2$ of the relative dualizing sheaf square is a locally free sheaf of the expected dimension - the direct image is the sheaf of holomorphic sections of a vector bundle.

First we observe that the dimension of the space of regular quadratic differentials on the fibers of $\piCB$ is constant.  The fibers of the family are compact noded Riemann surfaces; the distinguished points on the normalization of a fiber are inverse images of nodes.  A regular quadratic differential, equivalently a section of the dualizing sheaf square, is a meromorphic quadratic differential on each part with possible double poles at the inverse images of nodes and residues equal for each pair of nodal inverse images.  By Riemann Roch and the negative Euler characteristic hypothesis, for a part of genus $g$ with $n$ distinguished points the dimension of the space of quadratic differentials with possible double poles is $3g-3+2n$.  A node corresponds to a pair of distinguished points on the normalization and a pair of distinguished points contributes 
$4$ to the dimension of quadratic differentials before imposing the equal residue condition.   A node net contributes $3$ to the dimension; or equivalently each distinguished point on the normalization contributes a net $3/2$ to the dimension.  With this counting convention, the contribution to the dimension of regular quadratic differentials from a part of the normalization is $3g-3+3n/2$, the $-3/2$ multiple of the Euler characteristic of the part.  The Euler characteristic is additive - the sum of Euler characteristics of parts of a fiber coincides with the Euler characteristic of the general fiber. It follows that the dimension of the space of regular quadratic differentials on a fiber is a constant, equal to the $-3/2$ multiple of the Euler characteristic of the general fiber.  In the following, we write $\bg$ for the genus of the general fiber.

The following result on families of differentials is Proposition 5.1 of \cite{Msext} and Proposition 4.1 of \cite{HKmg}.  Masur proves the result by considering the sheaf of canonical forms on $\caC$ and considering the short exact sequences given by forming Poincar\'{e} residues onto successive hyperplane intersections with the family to define the fibers in $\caC$.  Direct image sheaves are formed.  Grauert's theorem \cite{Graucoh} is applied to show that the direct image sheaves are coherent.  Kodaira-Serre duality and a sequence chase are used to show that the Poincar\'{e} residue maps are surjective, his first result.  He applies the result to construct local sections of $\Pi_*\omCB$ with prescribed polar divisors.  A general  argument about products of $1$-forms is applied to find the desired families of regular quadratic differentials. We give a sketch of the much simpler Hubbard-Koch argument \cite{HKmg,HKlet}.    

\begin{lemma}\label{locfree} The direct image sheaf $\Pi_*\omCB^2$ is locally free rank $3\bg-3$.   A regular quadratic differential on a fiber of $\piCB$ is the evaluation of a section $\Pi_*\omCB^2$.
\end{lemma} 
\begin{proof}  Begin with a proper analytic map $\piCB$ of analytic spaces and a coherent sheaf $\caF$ on $\caC$.  For a point $b\in\caB$, consider the fiber $\caC_b=\Pi^{-1}(b)$ and $\caF_{\caC_b}$ the coherent sheaf on $\caC_b$ whose stalk at $c\in\caC_b$ is the tensor product $\caF_c\otimes_{\caO_{\caC}}\caO_{\caC_b}$ of the stalk of $\caF$ with the structure sheaves of $\caC$ and $\caC_b$.  By a form of the Cartan-Serre theorem with parameters the direct image of $\caF$ on $\caB$ is a locally free sheaf provided the dimensions of the cohomology groups 
$H^n(\caC_b,\caF_{\caC_b})$ are constant.  For the sheaf $\omCB^2$ the higher cohomology groups vanish and from the above discussion the zeroth cohomology has constant rank $3\bg-3$.  The direct image sheaf is locally free of rank $3\bg-3$.  The second statement of the lemma follows since the dimension of regular quadratic differentials equals the rank of the direct image sheaf. 
\end{proof} 

The general fiber  of $\piCB$ is a smooth Riemann surface with restrictions of sections of $\Pi_*\omCB^2$ to a fiber giving $Q$, the space of holomorphic quadratic differentials on the fiber.  By hypothesis a neighborhood of each node in $\piCB$ is analytically equivalent to a Cartesian product with $\piVD$.  The pairing of the infinitesimal variation of the parameter $t$, $t\ne 0$, with $Q$ is fundamental to our considerations.  For the pairing (\ref{Serrepair}), the following lemma is given as formula 7.1 in \cite{Msext}, formula (4) in \cite{Wlcomp} and in Proposition 9.7 in \cite{HKmg}.  The proof is to give a quasiconformal map from a $t$-fiber of $\piVD$ to a $t'$-fiber, compute the initial $t'$-derivative of the Beltrami differential and compute the pairing (\ref{Serrepair}) with a quadratic differential on the $t$-fiber. The quasiconformal map can be given to preserve concentric circles and as a rotation near the boundary.

\begin{lemma}\label{plumderiv}
The pairing of the infinitesimal $t$-variation, $t$ nonzero, with a section $\phi$ of the direct image sheaf $\Pi_*\omCB^2$  on the $t$-fiber is
\[
\big(\frac{\partial}{\partial t},\phi\big)\,=\,\frac{-\pi}{t}\Laur(\phi).
\]
\end{lemma}

We examine the geometric role of the $t$-scaling in the following sections. We now construct a general family for varying a nodal stable curve/a Riemann surface with nodes $\mathbf C$. In the next section, we describe the Hubbard-Koch Theorem for the universal property of the family.  Varying $\mathbf C$ is given as a combination of quasiconformal deformations of a relatively compact set in the normalization of $\mathbf C$ and gluings to copies of the family $\piVD$. The construction only uses that Beltrami differentials define finite deformations.  \v{C}ech style sliding deformations could equally be used in the construction in place of Beltrami differentials.

We start with the normalization of $\mathbf C$: a Riemann surface $R$ with formally paired distinguished points $\{a_k,b_k\}_{k=1}^m$ with corresponding local coordinates $z_k$ near 
$a_k$, $z_k(a_k)=0$, $w_k$ near $b_k$, $w_k(b_k)=0$ and with Beltrami differentials $\{\nu_j\}_{j=1}^n$ a basis for the cohomology group $H_{\bar\partial}^{0,1}(R,\caE(K^{-1}\prod_k a_k^{-1}b_k^{-1}))$.  We require that the domains of the local coordinates $\{z_k,w_k\}_{k=1}^m$ are mutually disjoint and that the union of local coordinate domains is disjoint from the union of Beltrami differential supports.  Let $\nu(s)=\sum_{j=1}^ns_j\nu_j,\,s\in\mathbb C^n,$ with $|s|$ small and let $R_s=R_{\nu(s)}$ be the deformation described above by quasiconformal maps.  The parameter $s$ is a local coordinate for the Teichm\"{u}ller space of $R$.  The coordinates $\{z_k,w_k\}_{k=1}^m$ are holomorphic for $R_s$, given the disjoint support hypothesis.  

We next plumb the family $R_s$ to $m$ copies of the family $\piVD$.  Choose a positive constant $c<1$, such that the local coordinate domains contain the discs $\{|z_k|<c\},\{|w_k|<c\},\, 1\le k\le m$.  For $|t|<c^4$, remove from $R_s$ the discs $\{|z_k|\le c^2\}$ and $\{|w_k|\le c^2\}$ to obtain an open surface $R^*_s$.  We now write $S$ for the connected domain of the parameter $s$ and $D$ for the disc $\{|t|<c^4\}$.  Boundary neighborhoods of individual fibers of 
$R_s^*\rightarrow S$ are glued to boundary neighborhoods of individual fibers of $\piVD$.  Given $(s,t)\in S\times D^m$, identify $p$ in the boundary neighborhood $\{c^2<|z_k(p)|<c\}\subset R_s^*$ to the point $(z_k(p),t_k/z_k(p))$ in the fiber of a $k^{th}$ factor of $\piVD$ and identify $q$ in the boundary neighborhood $\{c^2<|w_k(q)|<c\}\subset R_s^*$ to the point $(t_k/w_k(q),w_k(q))$ in the fiber of a $k^{th}$ factor of $\piVD$.  The result is a family of compact possibly noded Riemann surfaces $\piRS$ with projection - in brief the family given by a boundary neighborhood gluing of $R_s^*\rightarrow S$ and $m$ copies of the plumbing.  The plumbings determine labeled annuli in the fibers of the family.

\begin{figure}[htbp] 
  \centering
  \includegraphics[bb=0 0 527 222,width=4in,height=1.69in,keepaspectratio]{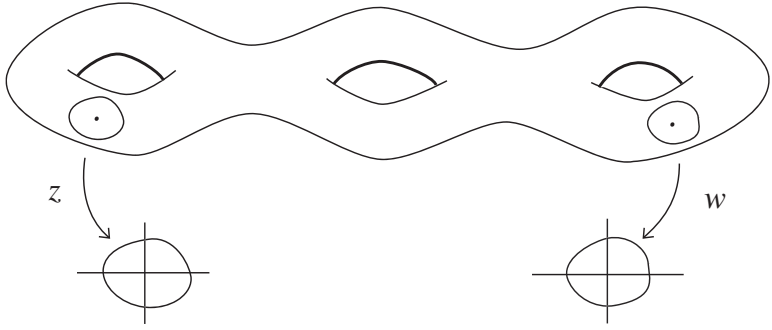}
  \caption{Plumbing data.}
  \label{fig:plumbdat}
\end{figure}

\begin{defn} 
The constructed family $\piRS$ is an $(s,t)$ family.
\end{defn}

The total space and base of an $(s,t)$ family are complex manifolds and the projection is  proper and a submersion except on the combined nodal set of codimension $2$.  By construction, a neighborhood of a node in an $(s,t)$ family is analytically equivalent to a Cartesian product of $\piVD$ and a complex manifold.  The dimension count for regular quadratic differentials provides for $\bg$ the genus of a smooth fiber that the dimensions satisfy $m+n=3\bg-3$.

\begin{lemma}\label{frameeval}
For an $(s,t)$ family $\piRS$, there is a neighborhood of the origin in $S\times D^m$ with sections $\phi_1,\dots,\phi_{3\bg-3}$ of the direct image sheaf $\Pi_*\omRS^2$ with the following pairings on $S\times\cap_k\{t_k\ne0\}$
\begin{align*}
\big(\frac{\partial}{\partial s_h},\phi_j\big)\,&=\,\delta_{hj},
&\big(\frac{\partial}{\partial s_h},\phi_{n+k}\big)\,&=\,0,\\
\Laur_k(\phi_h)\,&=\,0 \qquad\qquad\qquad\mbox{and}  &\Laur_k(\phi_{n+\ell})\,&=\,\frac{-\delta_{k\ell}}{\pi}
\end{align*}
for $1\le h,j\le n$, $1\le k,\ell\le m$ with $\delta_{**}$ the Kronecker delta and $\Laur_k$ the coefficient map for the $k^{th}$ annulus in a fiber.  On $S\times \cap_k\{t_k\ne0\}$ the restriction of the sections $\phi_1,\dots,\phi_n,t_1\phi_{n+1},\dots,t_m\phi_{n+m}$ to the fibers are the cotangents for the coordinates $(s,t)$.  
\end{lemma}
\begin{proof}  We begin by showing that the pairings of sections of the direct image sheaf with the tangents $\partial/\partial s_h$ and the coefficient maps $\Laur_k$ are holomorphic on a neighborhood of the origin in $S\times D^m$.  The initial Riemann surface $R$ has atlas $\{(U_{\alpha},z_{\alpha})\}$ and the surface $R_s$ has atlas $\{(U_{\alpha},w_{\alpha}\circ z_{\alpha})\}$ for $w_{\alpha}=w^{\nu(s)}$  solutions of $w_{\bar z}=\nu(s)w_z$ on $z_{\alpha}(U_{\alpha})\subset\mathbb C$.  The composition rule for quasiconformal maps provides that on $w_{\alpha}\circ z_{\alpha}(U_{\alpha})$ the Beltrami differential for the infinitesimal variation $\partial/\partial s_h$ is
\[
\mathbf L^{\nu(s)}\nu_h\,=\,\biggr(\frac{\nu_h}{1-|\nu(s)|^2}\frac{w^{\nu(s)}_z}{\overline{w^{\nu(s)}_z}}\biggr)\circ(w^{\nu(s)})^{-1},\ \ \cite{Ahsome}.
\]
For $|s|$ small, the Beltrami differentials $\nu_h,\,\mathbf L^{\nu(s)}\nu_h,\,1\le h\le n$ are bounded with compact support on the surface $R$.  A section $\phi$ of the direct image sheaf $\Pi_*\omRS^2$ has uniformly bounded $L^1$ integral on the support of the Beltrami differentials in the fibers.  It follows that the pairings $(\partial/\partial s_h,\phi)$ are bounded.  By the basics of Teichm\"{u}ller theory the pairings are holomorphic 
on $S\times\cap_k\{t_k\ne 0\}$.  By the Riemann extension theorem, the pairings are consequently holomorphic on the domain of $\phi$ in $S\times D^m$.  The discussion of sections of $\Pi_*\omVD^2$ provides that the coefficient maps $\Laur_k(\phi)$ are holomorphic on the domain of $\phi$.
By Lemma \ref{locfree}, there are sections $\psi_1,\dots,\psi_{3\bg-3}$ with the given pairings and Laurent coefficients at the origin $(s,t)=(0,0)$.  From the above, the matrix of evaluating pairings and Laurent coefficients is holomorphic and the identity at the origin. The matrix inverse applied to the sections $\psi_1,\dots,\psi_{3\bg-3}$ gives the desired sections with the desired pairings and Laurent coefficients.  The statement about cotangents follows immediately from Lemma \ref{plumderiv}.  
\end{proof} 

For all applications of Lemma \ref{frameeval}, we restrict the base of the family to provide that the sections $\phi_1,\dots,\phi_{3\bg-3}$ and relations are defined for the family.  
\begin{defn}
For an $(s,t)$ family $\piRS$ and the direct image 
sheaf $\Pi_*\omCB^2$, referring to Lemma \ref{frameeval}, the sections 
$\phi_1,\dots,\phi_{3\bg-3}$ are the first $(s,t)$ frame and the sections $\phi_1,\dots,\phi_n,t_1\phi_{n+1},\dots,t_m\phi_{n+m}$ are the second $(s,t)$ frame.
\end{defn} 
  
\section{Cotangents to stable curve moduli.}

We consider proper families of stable curves $\piCB$ including nodal curves with $\caC$ and $\caB$ smooth manifolds.  We assume a neighborhood of each node is analytically equivalent to a Cartesian product of $\piVD$ and a complex manifold parameter space.  The Cartesian product structure provides that the locus of nodes is a disjoint union of codimension $2$ smooth subvarieties.  A section of $\omCB^2$ has a canonical residue at a node; the residue is computed on the normalization.

We consider the vanishing residue subsheaf $\caV$ of $\omCB^2$ and the associated exact sequence.  The direct image sheaf $\Pi_*\caV$ is shown to be locally free and the sequence of direct images is exact.  The Hubbard and Koch characterization and description of a universal family are applied to show that the direct image $\Pi_*\caV$ is the cotangent sheaf for an admissible family.  Applications include that first $(s,t)$ frames are frames for the log-cotangent sheaf of the divisor of noded curves and  second $(s,t)$ frames are coordinate cotangent frames.  

\begin{defn}
For a family $\piCB$ as above. On $\caC$, the locus $\bn_k,1\le k\le m,$ is the $k^{th}$ component of the loci of nodes and $\mathbb C_{\bn_k}$ is the skyscraper sheaf with fiber $\mathbb C$ on $\bn_k$ and $0$ fiber on the complement. The vanishing residue subsheaf $\caV$  of $\omCB^2$ is defined by the exact sequence
\begin{equation}\label{vanseq}
0\longrightarrow\caV\longrightarrow\omCB^2\stackrel{\oplus_k\res_k}{\longrightarrow}\oplus_k\mathbb C_{\bn_k}\longrightarrow 0,
\end{equation}
where $\res_k$ is the residue on $\bn_k$ and the $0$ map on the complement. 
\end{defn} 

Although we do not provide the general argument, the relative dualizing sheaf, vanishing residue subsheaf and exact sequence are natural for analytic families $\piCB$.  On the complement of nodes in a family $\piCB$, the sheaves $\caV$ and $\omCB^2$ coincide; $\caV$ generalizes the notion of families of holomorphic quadratic differentials.  The expansion (\ref{fzws}) for sections of $\omVS^k$ shows that at a node the stalk of $\caV$ has rank $2$ over the structure sheaf with generators $z$ and $w$ in terms of the local coordinate $(z,w,s)$; the sheaf $\caV$ is not locally free.

\begin{lemma}
For an $(s,t)$ family $\piRS$, the elements of the second $(s,t)$ frame of Lemma \ref{frameeval} are sections of $\caV$.  The direct image $\Pi_*\caV$ is a locally free sheaf with basis the second $(s,t)$ frame.  The sequence of direct images is exact
\[
0\longrightarrow\Pi_*\caV\longrightarrow\Pi_*\omRS^2\stackrel{\oplus_k\res_k}{\longrightarrow}\oplus_k\Pi_*\mathbb C_{\bn_k}\longrightarrow 0.
\]
\end{lemma}
\begin{proof}
The first statement follows from Lemma \ref{frameeval}, noting that the fourfold multiple of the residue at a node is  the Laurent coefficient evaluated at $t=0$.  By Lemma \ref{locfree}, $\Pi_*\omRS^2$ is locally free rank $3\bg-3$ and from Lemma \ref{frameeval}, the first $(s,t)$ frame elements are linearly independent on fibers, and thus are independent over $\caO_{S\times D^m}$.  To consider direct images, we work with presheaves on $S\times D^m$.  Given the inclusion $\Pi_*\caV\subset\Pi_*\omRS^2$, an element $\phi$ of $\Pi_*\caV$ is a linear combination $\phi=\sum_{j=1}^{3\bg-3}f_j\phi_j$ for $f_j\in\caO_{S\times D^m}$ and $\phi_j$ the first $(s,t)$ frame.  On the locus of the $k^{th}$ node $\bn_k,1\le k\le m,$ the residues $\res_k(\phi_j),j\ne n+k$ vanish and $\res_k(\phi_{n+k})=-1/4\pi$.  
It follows that $f_{n+k}$ vanishes on the divisor $\Pi(\bn_k)$.  The coordinate $t_k$ vanishes to order one on the divisor; it follows that $f_{n+k}=t_k\tilde f_{n+k}$ for $\tilde f_{n+k}\in\caO_{S\times D^m}$.  The observation applies for each node and consequently $\phi$ is a combination over $\caO_{S\times D^m}$ of the second $(s,t)$ frame.  The direct image of $\caV$ is consequently locally free with basis the second $(s,t)$ frame.  

The first and second $(s,t)$ frames are respectively frames over $\caO_{S\times D^m}$ for sections of $\omRS^2$ and $\caV$ on $\Pi^{-1}(S\times D^m)$.  The constant section $1$ is a frame for $\mathbb C_{\bn_k}$ over $\caO_{S\times D^m}$ on $\bn_k$.  By Lemma \ref{frameeval}, the $(n+k)^{th}$ element of the first frame maps by the combined residue to the constant section $-1/4\pi$ of $\mathbb C_{\bn_k}$ and $0$ for the remaining skyscrapers.  The combined residue is surjective.  By the above argument, an element of the kernel of the combined residue map is a combination of second frame elements and the sequence is exact at $\Pi_*\omRS^2$.      
\end{proof} 

We use the work of Hubbard and Koch to relate $(s,t)$ families \cite{HKmg}.  Hubbard and Koch study proper flat families of nodal curves \cite[Definition 3.1]{HKmg}.
A proper surjective analytic map $\piCB$ of analytic spaces is a family of nodal curves provided at each point $\mathbf c\in\caC$, either $\Pi$ is smooth with one dimensional fibers or the family is locally isomorphic to the locus of $zw=f(s)$ in $\mathbb C^2\times S$ over $S$   with $f$ vanishing at the image of $\Pi(\mathbf c)$.  We work with the case of  first order vanishing of $f$, where families are locally analytically equivalent to the Cartesian product of $\piVD$ and a complex manifold \cite[Chapter 10, Proposition 2.1]{ACG2}.     
Families are further assumed to be $\Gamma$-marked, $\Gamma$ a multi curve on a reference surface  \cite[Definitions 5.1, 5.2 and 5.5]{HKmg}; a $\Gamma$-marking is a continuous section into a family of cosets of mapping classes from a reference surface to the fibers of the family.  They show in their Proposition 5.7 that a proper flat family $\piCB$ of stable curves has a $\Gamma$-marking if and only if there is a Hausdorff closed set $\caC'\subset\caC$ such that: i) each component of $(\caC-\caC')\cap\Pi^{-1}(b), b\in\caB$ is either an annulus or a one point union of discs (a degenerate annulus) and ii) $\Pi:\caC'\rightarrow\caB$ is a trivial bundle of surfaces with boundary.  The multi curve $\Gamma$ represents the annuli that can degenerate in the family.  

By construction $(s,t)$ families are proper flat with fibers stable curves with Hausdorff closet set the complement of the plumbing annuli.  An $(s,t)$ family is an example of their $\Gamma$-marked plumbed analytic family $\Pi:Y\G\rightarrow\caP\G$ \cite[Sections 8.2, 8.3]{HKmg}.  They also construct a standard model family $\Pi:X\G\rightarrow Q\G$ \cite[Definition 7.1, Section 7.1]{HKmg}. By their Theorem 9.11, the induced map from a neighborhood of the origin in the base of an $(s,t)$ family into $Q\G$ is a $\Gamma$-marking preserving topological equivalence of families.  
Their Theorem 10.1 provides that for an $(s,t)$ family $\piRS$, 
$\Gamma$-marked by plumbing annuli, there is a unique analytic map of $S\times D^m$ into $Q\G$ such that the $\Gamma$-marked family $\Pi:X\G\rightarrow Q\G$ pulls back to a $\Gamma$-marked family analytically equivalent to $\piRS$.  
In particular, $(s,t)$ families provide local analytic descriptions for the standard model family $\Pi:X\G\rightarrow Q\G$.  Furthermore, for a pair of $(s,t)$ families with a common fiber, there are neighborhoods of the base points and $\Gamma$-markings such that the families over the neighborhoods are uniquely $\Gamma$-marked equivalent under a biholomorphism. 

Basic considerations provide that standard model families $\Pi:X\G\rightarrow Q\G$ and $(s,t)$ families are locally {\em Kuranishi} \cite[Chapter 11, Section 4]{ACG2}.  A deformation 
 of a curve $\mathbf C$ is a family of noded curves $\piCB$ and basepoint $\mathbf b\in\caB$ with a prescribed analytic equivalence of $\mathbf C$ to the fiber $\Pi^{-1}(\mathbf b)$.  A deformation $\piCB$ of a curve $\mathbf C$ is Kuranishi provided for any deformation of the curve $\Pi:\caD\rightarrow\caE$ and basepoint $\mathbf e\in\caE$, that the family $\Pi:\caD\rightarrow\caE$ is locally at $\mathbf e$ the pullback of the family $\piCB$ locally at $\mathbf b$ by a unique morphism respecting the equivalences to $\mathbf C$.  Existence of a Kuranishi family for a curve is equivalent to stability of the curve \cite[Chapter 11, Theorem 4.3]{ACG2}.  In particular, a standard model family $\Pi:X\G\rightarrow Q\G$ with basepoint $\mathbf q\in Q_{\Gamma}$ is locally the pullback of a Kuranishi family.  At the same time, the Kuranishi family is locally $\Gamma$-marked by plumbing annuli.  The Hubbard and Koch Theorem 10.1 provides that the Kuranishi family is locally the pullback of the standard model family. In particular, the families are locally equivalent for neighborhoods of the designated fibers.   The local uniqueness property for the Kuranishi family transfers  to the standard model family.  Kuranishi families have many general properties \cite[Chapter 11, Section 4]{ACG2}.

\begin{defn}
An admissible family is a proper $\Gamma$-marked family $\piCB$ of stable curves with a neighborhood of each node analytically equivalent to a Cartesian product of $\piVD$ and a complex manifold.  
\end{defn}

Hubbard and Koch use their standard model family to define an analytic structure for the quotient of the augmented Teichm\"{u}ller space by local actions of mapping class groups \cite[Section 12]{HKmg}.  In Theorem 13.1, they show that the constructed space is analytically equivalent to the analytic form of the Deligne-Mumford compactification of the moduli space of Riemann surfaces.  In particular $(s,t)$ families provide local analytic descriptions for the local manifold covers of the Deligne-Mumford compactification.  

\begin{thrm}\label{main}
The vanishing residue sequence (\ref{vanseq}) is natural for admissible families.  The direct image of the vanishing residue sheaf is naturally identified with the cotangent bundle for admissible families.  
\end{thrm}
\begin{proof}
The proof is formal given the present considerations; we use simplified notation. We use local descriptions in terms of $(s,t)$ families.  For $\Gamma$-marked families $\piRS$ and  $\Pi':\caR'\rightarrow S'\times D'^{m'}$ with a common fiber, there is a 
biholomorphic equivalence of the families over neighborhoods of the common base point; 
restrict the families to the neighborhoods.  If the number of nodes on the common fiber is not maximal for the families then $m$ and $m'$ can be distinct.  The relative dualizing sheaf is characterized either by its dualizing property or as the dual of the vertical line bundle.  The nodal loci are characterized by self crossing of the fiber of the projection and the residue on the normalization is intrinsic.  A biholomorphism $F$ pulls back the second family to the first family.  The vanishing residue exact sequence for the second family pulls back to the vanishing residue exact sequence for the first family.  On the domain of $F$ the coordinate cotangents are related by
\[
(ds',dt')\,=\, \frac{\partial(s',t')}{\partial (s,t)}(ds,dt)^t
\]
for an analytic Jacobian.  By Lemma \ref{frameeval}, on the open set $S\times\cap_{k=1}^m\{t_k\ne 0\}$ the cotangents are represented by the second $(s,t)$ frames $(\phi,t\phi)$ for $(ds,dt)$ and $(\phi',t'\phi')$ for $(ds',dt')$.  On the open set the second frames satisfy the relation
\[
(\phi',t'\phi')\,=\,\frac{\partial(s',t')}{\partial (s,t)}(\phi,t\phi).
\]
The left and right hand sides independently give $3\bg-3$ analytic sections of $\Pi'_*\caV'$ and $\Pi_*\caV$.  The relation is consequently satisfied on the domain of $F$.  Since $(\phi',t'\phi')$ and $(\phi,t\phi)$ are local bases for the locally free sheaves $\Pi_*'\caV'$ and $\Pi_*\caV$ and $\partial (s',t')/\partial (s,t)$ is the cocycle for the cotangent bundle, it follows that the direct image of $\caV$ is analytically equivalent to the cotangent bundle.  The equivalence does not depend on the particular representation as an $(s,t)$ family. For an admissible family the equivalence is given by any local representation.
\end{proof}

For an admissible family the form for a local frame of $\Pi_*\caV$ is determined by general considerations.   Begin with the restriction map from the relative dualizing sheaf to the dualizing sheaf of a fiber.  By the dimension count leading to Lemma \ref{locfree}, for a noded fiber the space of sections of the dualizing sheaf square vanishing at nodes has codimension the number of nodes.  By the same count, for each node there is a local section of the dualizing sheaf square with unit residue at the node and vanishing at other nodes.  Now the product of a function defining the associated locus of nodal curves in the base and a section of the relative dualizing sheaf square with constant unit residue is a section of $\Pi_*\caV$ corresponding to the node.  As sections of $\Pi_*\caV$, the constructed products are locally independent and independent from sections with non trivial restrictions to the fiber.  The collection of sections provides a local frame for $\Pi_*\caV$.  We will see below that a section of $\Pi_*\caV$ vanishing on noded fibers is a conormal to the locus of noded curves.     

\begin{corollary}\label{cotanframe}
For an $(s,t)$ family the second $(s,t)$ frame is the coordinate cotangent frame. 
\end{corollary}
 
To understand the description of the cotangent fibers, we review the isomorphism between locally free sheaves and analytic vector bundles. Associated to an analytic vector bundle $E$ over a connected complex manifold $M$ is the sheaf $\caO(E)$ of analytic sections.  The sheaf is locally free.  In particular if $E$ is trivial over an open set $U$, the space of sections over $U$ is isomorphic to $\caO(U)^r$ for $r$ the rank of $E$.  Conversely, if for an open cover $\{U_{\alpha}\}$ of $M$, the space of sections of a sheaf $\caE$ 
is given as locally free by isomorphisms $h_{\alpha}:\Gamma(U_{\alpha},\caE)\rightarrow\caO(U_{\alpha})^r$, then the cocycles $h_{\alpha\beta}=h_{\alpha}\circ h_{\beta}^{-1}$ define invertible maps of $\caO(U_{\alpha}\cap U_{\beta})^r$ with the compatibility condition $h_{\alpha\beta}\circ h_{\beta\gamma}=h_{\alpha\gamma}$ on $U_{\alpha}\cap U_{\beta}\cap U_{\gamma}$. The cocycle $\{h_{\alpha\beta}\}$ defines a rank $r$ vector bundle on $M$ \cite[Chap. II]{Wells}. 

Given a locally free sheaf $\caE$, the fibers of the corresponding vector bundle are described by an elementary sheaf construction.  For the fiber at $p$ in $M$, introduce the sky-scraper sheaf $\mathbb C_p$ supported at $p$; a complex number defines a section of $\mathbb C_p$ and $\mathbb C_p$ is an $\caO_M$ module by evaluating functions at $p$.  The vector bundle associated to $\caE$ has fiber at $p$ the cohomology group $H^0(M,\caE\otimes_{\caO_M}\mathbb C_p)$.  A germ at $p$ of a section of $\caE$ defines an element of the fiber by extending by zero for the complement of $p$.  Germs $\mathbf e$ and $\mathbf e'$ at $p$ of sections of $\caE$ define the same element of the fiber provided for representatives on a neighborhood of $p$, that $\mathbf e=f\mathbf e'$ for $f$ in $\caO(U)$ with $f(p)=1$ or equivalently $\mathbf e-\mathbf e'=h\mathbf e''$ for $\mathbf e''$ a representative of a germ of a section and $h$ in $\caO(U)$ with $h(p)=0$. 

An example of the fiber construction is for the rank one sheaf $\caO(p)$ for a point on a Riemann surface.  For germs of sections $\sigma$ and $\sigma'$ at $p$, each with divisor $p$, then the germs satisfy $\sigma=f\sigma'$ for $f$ a germ of $\caO$ at $p$.  In the $p$-fiber $H^0(\caO(p)\otimes_{\caO}\mathbb C_p)$, the class of $\sigma$ is the $f(p)$ multiple of the class of $\sigma'$. 

We apply the elementary sheaf construction to describe the fibers of the cotangent bundle of an $(s,t)$ family.  

\begin{lemma}\label{cotfib}
Let $\piRS$ be an $(s,t)$ family with second $(s,t)$ frame $\caF$ and $\caI_p\subset\caO_{S\times D^m}$, the ideal sheaf of a point in the base.  The cotangent space of $S\times D^m$ at $p$ is identified with $\spn_{\caO_{S\times D^m}}(\caF)/\spn_{\caI_p}(\caF)$.  There is a well defined pairing between Beltrami differentials supported away from the nodes on the fiber $\Pi^{-1}(p)$, the functionals $-\pi/t_k\Laur_k$ and the cotangent vectors at $p$ given by forming limits with representatives. 
\end{lemma}
\begin{proof}
A frame for the cotangent sheaf of the $(s,t)$ family is given by the coordinate frame $\caF$.  The description of the cotangent fiber at $p$ comes from the above sheaf description of a vector bundle fiber.  From the proof of Lemma \ref{frameeval}, the pairings of the tangent functionals with $\caF$ are holomorphic on a neighborhood of the origin in $S\times D^m$.  It follows that elements in $\spn_{\caI(p)}\caF$ have pairings vanishing at $p$.  The conclusion follows. 
\end{proof}

We can describe the cotangent fiber at a point for the curve $\mathbf C$ by giving sections of the dualizing sheaf square and local coordinates at the nodes.  The normalization of $\mathbf C$ is a Riemann surface $R$ with formally paired distinguished points $\{a_k,b_k\}_{k=1}^m$. Assume that local coordinates $z_k$ near $a_k$, $z_k(a_k)=0$ and $w_k$ near $b_k$, $w_k(b_k)=0$ are given.  Consider $Q(R)$ the space of meromorphic quadratic differentials on $R$ with at most simple poles at the distinguished points $\{a_k,b_k\}_{k=1}^m$; $Q(R)$ is the cotangent space of the Teichm\"{u}ller space of $R$.  Given an $(s,t)$ family $\piRS$ containing a fiber isomorphic to $\mathbf C$, we describe sections of $\Pi_*\caV$ representing the cotangents at $\mathbf C$.  
By Lemma \ref{locfree}, each section $\beta$ of the dualizing sheaf square is the evaluation of a section $\tilde\beta$ of the direct image $\Pi_*\omRS^2$.
The $(s,t)$ family is equivalent to a family with local coordinates $\{z_k,w_k\}_{k=1}^m$ for plumbing data.  The coordinates define sheaf maps $\Laur_k:\omRS^2\rightarrow \caO_{S\times D^m}$.
By Lemma \ref{frameeval}, there are also sections $\tilde\beta_h$, $1\le h\le m$, of the direct image with $\zeta$-coefficients $\Laur_k(\tilde\beta_h)=\delta_{kh}$, $1\le k,h\le m$, for $\delta_{**}$ the Kronecker delta.  
For $\beta\in Q(R)$, the system of equations $\Laur_k(\tilde\beta+\sum_{h=1}^mb_h\tilde\beta_h)=0$, $1\le k,h\le m$, for functions $b_h\in\caO_{\caB}$, has a local solution.  
Solving the system defines a linear map depending on the data of local coordinates: for $\beta\in Q(R)$, associate $\caL_{\{z_k,w_k\}}(\beta)=\tilde\beta+\sum_{h=1}^mb_h\tilde\beta_h$, a section of $\Pi_*\caV$. Furthermore for $t_h=z_hw_h$, the products $-t_h\tilde\beta_h/\pi$, $1\le h\le m$, define sections of $\Pi_*\caV$ vanishing on the fiber isomorphic to $\mathbf C$.  By Lemma \ref{cotfib}, the pairings at $\mathbf C$ of the elements $\caL_{\{z_k,w_k\}}(\beta), -t_h\tilde\beta_h/\pi$, $\beta\in Q(R)$, $1\le h\le m$, with Beltrami differentials supported away from the nodes and with the infinitesimal $t_k$-variations are determined completely by the elements of $Q(R)$ and the local coordinates $\{z_k,w_k\}_{k=1}^m$.  We will see in Theorem \ref{compplum} that the cotangent vectors depend on the second order expansions of the local coordinates.
 
The local Cartesian product description at a node provides that for an admissible family $\piCB$, the locus $\caD\subset\caB$ of noded curves is a divisor with normal crossings. In general for a normal crossing divisor in a complex manifold $M$, there are local analytic coordinates $(z_1,\dots,z_{n+m})$ with $\caD=\cup_{k=1}^m\caD_k$ for $\caD_k=\{z_{n+k}=0\}$.  We are interested in three cotangent constructions associated to the divisor $\caD$ of noded curves.  The first is the dual of the normal sheaf of $\caD$, the conormal $N_{\caD}^{\vee}$, defined by the exact sequence
\[
0\longrightarrow N_{\caD}^{\vee}\longrightarrow\Omega_M\otimes\caO_{\caD}\longrightarrow\Omega_{\caD}\longrightarrow 0,
\]
for $\Omega_M$ the sheaf of $1$-forms on $M$ and $\Omega_{\caD}$ the sheaf of K\"{a}hler differentials for $\caD$ \cite{Bard}.  The conormal is the set of cotangent vectors to $M$ supported on $\caD$ and is the annihilator of the tangents to the intersection of $\caD$ components.  For example, at a point of $\cap_{k=1}^m\caD_k$, the conormal is the annihilator of the intersection $\cap_{k=1}^mT_{\caD_k}$ of tangent spaces.  At the same point, the conormal sections are generated by the differentials $dz_{n+k},1\le k\le m,$ and the K\"{a}hler differentials are generated by $dz_j, 1\le j\le n$.  The second cotangent construction is the log-cotangent bundle $\Omega_M(\log\caD)$ for a normal crossing divisor.  The log-cotangent is a subsheaf of the sheaf of meromorphic $1$-differentials with $\Omega_M(\log\caD) |_{M-\caD}$ the sheaf of analytic $1$-differentials.  At the above point of $\caD$, the sheaf is locally generated by the differentials $dz_j,\,dz_{n+k}/z_{n+k},\,1\le j\le n, 1\le k\le m$.  The log-cotangent bundle is locally free of rank $\dim M$.  The third cotangent construction, supported on $\caD$, is the space of cotangents represented on the normalizations of the stable curves by holomorphic quadratic differentials; the differentials are holomorphic at the nodal inverse images. In Theorem \ref{compplum}, we show that the space of such cotangents is the annihilator of the infinitesimal plumbing tangents.  

For an $(s,t)$ family $\piRS$ with parameters $(s_j,t_k), 1\le j\le n, 1\le k\le m,$ the divisor $\caD$ of noded curves is $\cup_{k=1}^m\caD_k$ with $\caD_k=\{t_k=0\}$.  The locus $\caD$ is the codimension $m$ locus with parameters $(s_j,0), 1\le j\le n$.  

\begin{corollary}\label{logpolframe}
For an $(s,t)$ family the log-cotangent bundle is generated by the first $(s,t)$ frame. The conormal of $\caD$ is generated by the elements $t_k\phi_k$ of the second $(s,t)$ frame. 
\end{corollary}
\begin{proof}
The statements are direct consequences of Corollary \ref{cotanframe} and Lemma \ref{frameeval}.
\end{proof}

\section{Comparing infinitesimal plumbings.}

We generalize the Laurent coefficient map of Definition \ref{Laurdef} for elements of $\omVD^2$. Evaluating the limit of the generalization will not require a special choice of coordinates.  The generalization enables comparison of infinitesimal plumbings and extends period expansions of quadratic differentials \cite{Gardtheta,Hejmono,HKmg,HSS, Kracusp,Msext,McM, HPZur,HPEin,Wlcbms,Wlcurv} to the limiting case of a nodal fiber. 

For the family $\piVD$, the vertical line bundle $\caL$ is dual to the relative dualizing sheaf $\omVD$.  For a section $\lambda$ of $\caL$, nonzero at the node, the adjoint residue value, defined modulo a sign, is the residue of the dual section $\lambda^{\vee}$ at the node.  The $t\ne0$ fibers of the family $\piVD$ are annuli.  For $\eta$ a section of $\omVD$ and $\gamma$ a loop in a fiber winding once around the annulus, then the integral $\int_{\gamma}\eta$ is the $2\pi i$ multiple of the zeroth Laurent coefficient of the restriction of $\eta$ to the fiber with the sign of the integral uniquely determined by the orientation of $\gamma$.

\begin{figure}[thb] 
  \centering
  \includegraphics[bb=20 118 575 673,width=2.5in,height=2.5in,keepaspectratio]{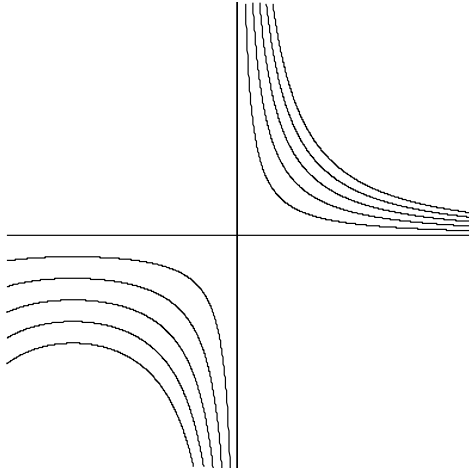}
  \caption{The family of hyperbolas $(z+.7z^2)w=t$. }
  \label{fig:hyperb2}
\end{figure}

\begin{defn} Given $\lambda$ a section of $\caL$ with adjoint residue value $\pm1$, a collection of loops $\gamma_t$ in the fibers $\pi^{-1}(t),t\ne0$, is admissible provided $\lim_{t\rightarrow0}\int_{\gamma_t}\lambda^{\vee}=4\pi i$.  For $\phi$ a section of $\omVD^2$ and $\lambda$ with admissible loops $\gamma_t$, the Laurent pairing is 
\[
\LP(\lambda,\phi)\,=\,\frac{1}{\pi i}\int_{\gamma_t}\,\lambda\phi.
\]
\end{defn}  
We consider properties of the pairing.  The product $\lambda\phi$ is a section of $\omVD$ and the integrals only depend on the homology classes of the loops.  The integrals do not depend on a choice of coordinates.  The pairing is analytic in $t$ and for $\phi$ a section of the vanishing residue sheaf $\caV$, then $\LP(\lambda,\phi)$ vanishes at $t=0$.  In particular for $\phi$ a section of $\caV$ the limit
\[
\lim_{t\rightarrow0}\frac{\LP(\lambda,\phi)}{t}
\]
exists.  

To find additional properties of the limit, we introduce expansions for $\lambda$ and $\phi$ and calculate.  For the local Cartesian product description $V\times S$, the considerations will only involve continuity in $S$ and so we omit $S$ dependence in the expansions.  We begin with a section of $\omVD$ and its dual section of $\caL$.  
From Section \ref{locgeom}, for an annulus $A$ with variable $\zeta$ mapped into a fiber of $V$ by $(\zeta,t/\zeta)$, then a $1$-differential $\eta=\mathbf f(z,w)(dz/z-dw/w)$ pulls back to 
$\mathbf{f}(\zeta,t/\zeta)2d\zeta/\zeta$ on $A$.  For a loop $\gamma$ in $A$ with $\int_{\gamma}d\zeta/\zeta=2\pi i$, the integral $\int_{\gamma}\mathbf{f}(\zeta,t/\zeta)2d\zeta/\zeta$ equals the $4\pi i$ multiple of the $\zeta$-constant of $\mathbf f$; see formula (\ref{zconst}).  Provided $\mathbf{f}(0,0)=1$, then in a neighborhood of the origin $(2\mathbf{f}(z,w))^{-1}(z\partial/\partial z\,-\,w\partial/\partial w)$ is an analytic section of $\caL$ with dual section $\eta$ satisfying the hypothesis $\lim_{t\rightarrow0}\int_{\gamma}\eta=4\pi i$.

We use the coordinate  $(z,w)$ for $V$ to give expansions for $\lambda$ a section of $\caL$ and $\phi$ a section of $\omVD^2$:
\[
\lambda\,=\,(1+az+bw+O_2)\frac12\big(z\frac{\partial}{\partial z}-w\frac{\partial}{\partial w}\big)
\]
and
\[
\phi\,=\,(cz+dw+ez^2+fzw+gw^2+O_3)\dzdw^2,
\]
where $O_r$ represents degree $r$ terms and higher in the variables.  The product has the expansion
\[
\lambda\phi\,=\,\big((ad+bc+f)zw+(ac+e)z^2+(bd+g)w^2+O_3\big)\dzdw
\]
and the pullback to the annulus $A$ has the expansion
\[
\lambda\phi\,=\,\big((ad+bc+f)t+O_1(\zeta)+O_1(1/\zeta)+O_2(t)\big)2\frac{d\zeta}{\zeta},
\]
and the Laurent pairing is
\begin{equation}\label{LPform}
\LP(\lambda,\phi)\,=\,4(ad+bc+f)t+O_2(t).
\end{equation}
In particular the coefficient of $t$ is determined by the first order expansion of $\lambda$ at the node and the second order expansion of $\phi$.  The first order expansion of the section $\lambda$ is determined by the restriction of $\lambda$ to the $t=0$ fiber.   In the particular case $\lambda=1/2(z\partial/\partial z-w\partial/\partial w)$ the Laurent pairing coincides with the Laurent coefficient, specifically $\LP(\lambda,\phi)=\Laur(\phi)$; compare the expansion (\ref{LPform}) to formula (\ref{zconst}).

In \cite{Wlcut} it is shown that the coefficient functionals $\phi\rightarrow\pi c(\phi)$, $\phi\rightarrow\pi d(\phi)$ are the evaluations of the {\em infinitesimal origin slidings} for the coordinates $z,w$ of the inverse images of the node on the normalization; the functionals are the action of the infinitesimal displacement of the location of possible poles for quadratic differentials on the normalization.  In the definition of sliding \cite[pgs. 405, 406]{Wlcut},  the integral is for a negatively oriented loop about the distinguished point and as noted on \cite[pg. 409]{Wlcut} the factor $i/2$ is omitted from the Beltrami-quadratic differential pairing; compare to the present formula (\ref{Serrepair}).  The infinitesimal plumbing formula 
below agrees\footnote[1]{There is a simple mistake in the referenced formula. On page 409 of the reference the function $G$ is normalized to $G'(0)=1$.  The ensuing calculation on page 411 refers to the normalized function.  In the statement of the Corollary the intended second derivative refers to the normalized derivative $(G'(0))^{-1}G''(0)$.} 
with the earlier formula \cite[Corollary]{Wlcut} obtained by constructing a two-parameter family interpolating between a pair of plumbings.

The $(s,t)$ coordinates depend on choosing plumbings.  The following formula gives the initial effect of a change of plumbing data.   To apply the formula for the second $(s,t)$ frame of Lemma \ref{frameeval}, note that in a neighborhood of the $k^{th}$ node, $t_k\phi_{n+k}=z_kw_k\phi_{n+k}$ for the second frame element $dt_k$.  

\begin{thrm}\label{compplum}
Let $\piCB$ be a proper family of stable curves with a nodal fiber $\caC_b$ with a neighborhood of a node analytically equivalent to the Cartesian product of $\piVD$ and a parameter space $S$.  Consider the local coordinates $F(z)$ and $G(w)$, $F(0)=G(0)=0$, for neighborhoods of the inverse images of the node on the normalization of $\caC_b$.  For the plumbing $F(z)G(w)=\tau$  of $\caC_b$ contained in $\piCB$ and $\phi$ a section of $\caV$ on a neighborhood of $\caC_b$, the initial plumbing tangent evaluates as
\begin{multline*}
\big(\frac{\partial}{\partial \tau},\phi\big)\Big|_{\tau=0}\,=\\
\pi(F'(0)G'(0))^{-2}\big(-F'(0)G'(0)\mathbf f_{zw}(0,0,s)\,+\,\frac12F''(0)\mathbf f_w(0,0,s)\,+\,\frac12G''(0)\mathbf f_z(0,0,s) \big)
\end{multline*}
for 
\[
\phi\,=\,\mathbf f(z,w,s)\dzdw^2
\]
and the given partial derivatives.
\end{thrm}  
\begin{proof}
The essential matter is a change of coordinates for a neighborhood of a node represented in two ways as a product with $V$.   If a neighborhood of the node is given as the Cartesian product of the plumbing $F(z)G(w)=\tau$ and a parameter space, then the initial plumbing tangent is evaluated by Lemma \ref{plumderiv} and a limit.  As noted above, for $(u,v)$ coordinates for $V$, the evaluation is equivalently given by the Laurent pairing with the section 
$1/2(u\partial/\partial u-v\partial/\partial v)$ of $\caL$ and a limit.  The section has dual section $\eta=du/u-dv/v$ of $\omCB$.  We now calculate the expansion of $\eta$ for a second description in terms of $V$ with coordinates $(z,w)$.  
In general for a change of variable $\zeta=f(z)$, $f(0)=0,$ then 
$d\zeta/\zeta=(1+(2f'(0))^{-1}f''(0)z+O_2)dz/z$.  As noted above, the first order expansion of $\eta$ is determined by its restrictions to the coordinate axes in $V$. On the $z$-coordinate axis, respectively $w$-coordinate axis, in $V$ the variable $w$, respectively $z$, and its differential vanish. Apply the general change of variable for the coordinate axes.  We have the following relation on the coordinate axes in $V$
\[
\big(\frac{du}{u}-\frac{dv}{v}\big)\,=\,\big(1+\frac12(F'(0))^{-1}F''(0)z+\frac12(G'(0))^{-1}G''(0)w+O_2\big)\dzdw.
\]
The dual section of $\caL$ satisfies
\[
\eta^{\vee}\,=\,\big(1-\frac12(F'(0))^{-1}F''(0)z-\frac12(G'(0))^{-1}G''(0)w+O_2\big)\frac12\big(z\frac{\partial}{\partial z}-w\frac{\partial}{\partial w}\big).
\]
We apply Lemma \ref{plumderiv} to find that the initial plumbing tangent is given as the limit
$\lim_{\tau\rightarrow 0}-\pi\LP(\eta^{\vee},\phi)/\tau$.  We apply formula (\ref{LPform}) and the observation $\lim_{\tau\rightarrow 0}t/\tau=\lim_{(z,w)\rightarrow(0,0)}zw(F(z)G(w))^{-1}=(F'(0)G'(0))^{-1}$ to obtain the final formula.  
\end{proof}

An easy application is the standard formula for the normal sheaf to the divisor $\caD$ of noded curves \cite[Props. 3.31, 3.32]{HMbook}.  The normal sheaf is the quotient of the tangent sheaf by the intersection of the tangent sheaves of components of $\caD$.  For an $(s,t)$ family $\piRS$ with parameters $(s_j,t_k),1\le j\le n, 1\le k\le m$, the intersection of components of $\caD$ has parameters $(s_j,0),1\le j\le n$.  The normal sheaf is the quotient of the $t$ parameter tangents modulo the $s$ parameter tangents.  The infinitesimal sliding deformations are represented by compactly supported Beltrami differentials on the normalization of a curve.  In particular, the infinitesimal slidings are tangents to the intersection of components of $\caD$. Theorem \ref{compplum} provides the relation 
\[
\frac{\partial}{\partial\tau}\,\equiv\, (F'(0)G'(0))^{-1}\frac{\partial}{\partial t}\mod\cap_{k=1}^m T_{\caD_k}
\]
between the plumbings $F(z)G(w)=\tau$ and $zw=t$ of a node.  The factor $(F'(0)G'(0))^{-1}$ is the cocycle for the tensor product of tangent lines to the normalized curve at the inverse images of the node.  In particular, the normal line is the product of tangent lines. 

Hyperbolic metrics provide a norm for the limit pairing $\lim_{t\rightarrow 0}\mathbf{LP}(\lambda,\phi)/t$ as follows.  For an $(s,t)$ family $\piRS$ with all $t$ parameters nonzero, the fibers are compact Riemann surfaces with uniformization hyperbolic metrics depending smoothly on the parameters.  For negative curvature metrics, there is a unique geodesic in each free homotopy class of closed curves.  The geodesic in the class of the $k^{th}$ annulus has length expansion
\[
\ell_k\,=\,\frac{2\pi^2}{\log1/|t_k|}\,+\,O((\log 1/|t_k|)^{-r}),
\]
locally uniformly in $s$ with $r=4$ for a single plumbing and $r=2$ in general \cite[Example 4.3]{Wlhyp}.  For a single plumbing, the length expansion provides that $e^{2\pi^2/\ell}\,=\,1/|t|(1+O((\log1/|t|)^{-2}))$.  We can define a norm for a single plumbing pairing by considering for $\lambda$ a section of $\caL$ and $\phi$ a section of $\Pi_*\caV$ the limit
\[
\lim_{t\rightarrow 0}e^{2\pi^2/\ell}\,|\mathbf{LP}(\lambda,\phi)|.
\]
For $\lambda$ with adjoint residue value $\pm1$ and $\phi$ vanishing on the noded fiber, the definition provides a norm for the conormal to the divisor of noded curves.

\section{Applications.}

\subsection{The moduli action of the automorphism group of a stable curve.}

We consider that $(s,t)$ families provide local manifold covers for the Deligne-Mumford compactification and describe the local actions of the relative mapping class groups.  We begin with the Hubbard-Koch discussion of markings \cite[Sections 2, 5]{HKmg}.  For $\Gamma$ a labeled multi curve on a compact reference surface $S$, let $\modd(S,\Gamma)$ be the group of mapping classes fixing the free homotopy classes of $\Gamma$ and the components of $S-\Gamma$.  Let $\modd_{ext}(S,\Gamma)$ be the extension of $\modd(S,\Gamma)$ of mapping classes permuting the $\Gamma$ free homotopy classes and components of $S-\Gamma$.  For $S/\Gamma$ the identification space given by collapsing $\Gamma$, let $\modd(S/\Gamma)$ be the group of isotopy classes of homeomorphisms of $S/\Gamma$ that fix each individual collapsed curve and fix the components of $S/\Gamma-\Gamma$.  Let $\modd_{ext}(S/\Gamma)$ be the extension of isotopy classes of homeomorphisms of $S/\Gamma$.  

For every multi curve subset $\Gamma'\subset \Gamma$, define $\operatorname{Homeo}(S,\Gamma,\Gamma')$ to be the homeomorphisms of $S-\Gamma'$, stabilizing each component of $S-\Gamma'$ and homotopic to a composition of Dehn twists about elements of $\Gamma-\Gamma'$.  Associated to an admissible family $\piCB$ are local families over $\caB$ of equivalence classes of maps from $S$ to the fibers of $\caC$.  In particular for $b\in\caB$, consider the maps of $S$ to $\Pi^{-1}(b)$ taking elements of a multi curve $\Gamma'$ bijectively to the nodes on the fiber.   A pair of such maps $\phi_1,\phi_2$ are {\em equivalent} provided for $ h\in\operatorname{Homeo}(S,\Gamma,\Gamma')$ that $\phi_1$ is homotopic to $ h\circ\phi_2$ on $S-\Gamma'$ by a homotopy of homeomorphisms from $S-\Gamma'$ to $\Pi^{-1}(b)-\{nodes\}$.  The set of equivalence classes of maps is $\markk_{\caC}^{\Gamma}(S;\caB)_b$.  The union 
$\markk_{\caC}^{\Gamma}(S;\caB)=\cup_{b\in\caB}\markk_{\caC}^{\Gamma}(S;\caB)_b$ inherits a topology by considering the compact-open topology for maps of $S$ into $\caC$.  The space $\markk_{\caC}^{\Gamma}(S;\caB)$ has a projection $p_{\Gamma}$ to $\caB$.  A $\Gamma$-marking for a family $\piCB$ is defined as a continuous section of the projection.  The mapping class group $\modd_{ext}(S,\Gamma)$ acts on $\Gamma$-markings by a right action of precomposition of maps.  In analogy to the local topological triviality of families of compact surfaces, Hubbard-Koch show that the projection $p_{\Gamma}$ has discrete fibers and a local section through each point \cite[Theorem 5.4]{HKmg}.  A direct characterization of $\Gamma$-markable families is provided \cite[Proposition 5.7]{HKmg}. 

Let $\Pi:X_{\Gamma}\rightarrow Q_{\Gamma}$ with $\Gamma$-marking $\phi_{\Gamma}$ be the Hubbard-Koch universal $\Gamma$-marked family.   
The natural map from a connected open set in the base of an $(s,t)$ family into the base $Q_{\Gamma}$ provides a $\Gamma$-marking preserving analytic equivalence of the $(s,t)$ family to the pullback family.  The universal property provides that a mapping class $\mathbf h\in\modd_{ext}(S,\Gamma)$ which acts by a right action by precomposition on $\Gamma$-markings, acts by a left action on the base $Q_{\Gamma}$ by an analytic equivalence $h_Q$ and on the total space by an analytic equivalence $h_X$  to give the family $h_Q\Pi h_X^{-1}:h_X(X_{\Gamma})\rightarrow h_Q(Q_{\Gamma})$ with marking $h_X\circ\phi_{\Gamma}\circ h^{-1}_Q$ \cite[Theorem 10.1]{HKmg}.  If each neighborhood of a basepoint $b$ of the family intersects its $h_Q$ translate, then $h_Q$ fixes $b$ and the mapping class $\mathbf h$ is realized by an automorphism acting on the $b$ fiber \cite[Proposition 2.6]{HKmg}.  Conversely, automorphisms of a fiber act by a left action on maps of $S$ to the fiber.  By discreteness and existence of $\markk^{\Gamma}_{X_{\Gamma}}(S;Q_{\Gamma})$ marking sections, automorphisms of a fiber of $\Pi:X_{\Gamma}\rightarrow Q_{\Gamma}$ locally act through the mapping class action. 

For a stable curve $\mathbf C$, consider a $\Gamma$-marked family $\Pi:X_{\Gamma}\rightarrow Q_{\Gamma}$ containing a fiber isomorphic to $\mathbf C$.  As described above, the automorphism group $\operatorname{Aut}(\mathbf C)$ acts on a neighborhood of the point in $Q_{\Gamma}$ representing $\mathbf C$.  

\begin{thrm}\textup{\cite[Proposition 12.2, Corollary 12.3]{HKmg},}\label{DM}
The Deligne-Mumford moduli space is an analytic orbifold.  A neighborhood of the stable curve $\mathbf C$ in the moduli space is locally modeled by the quotient $Q_{\Gamma}/\operatorname{Aut}(\mathbf C)$.
\end{thrm}
Let the normalization of the curve $\mathbf C$ be a Riemann surface $R$ with paired distinguished points $\{a_k,b_k\}_{k=1}^m$ with local coordinates $z_k$ near $a_k$, $z_k(a_k)=0$ and $w_k$ near $b_k$, $w_k(b_k)=0$.  The local coordinates define sheaf maps $\Laur_k:\omega^2_{X_{\Gamma}/Q_{\Gamma}}\rightarrow \caO_{Q_{\Gamma}}$.  
By Lemma \ref{frameeval}, there are sections $\tilde\beta_h$, $1\le h\le m$, of the direct image $\Pi_*\omega^2_{X_{\Gamma}/Q_{\Gamma}}$ with coefficients $\Laur_k(\tilde\beta_h)=\delta_{kh}$, $1\le k,h\le m$, for $\delta_{**}$ the Kronecker delta.
In the discussion following Lemma \ref{cotfib}, the cotangent space at $\mathbf C$ is described in terms of: $Q(R)$, the meromorphic quadratic differentials with at most simple poles at distinguished points, the local coordinates $\{z_k,w_k\}_{k=1}^m$ and the linear map $\caL_{\{z_k,w_k\}}$, extending sections of the dualizing sheaf square.  

\begin{corollary}\label{aut}
With the above conventions, $\operatorname{Aut}(\mathbf C)$ acts by a right action on the cotangent fiber at $\mathbf C$ by: the map $h_X$ pulls back $Q(R)$, pulls back the local coordinates $\{z_k,w_k\}_{k=1}^m$ and pulls back the linear map $\caL$. For $\operatorname{Aut}(\mathbf C)$ acting on local coordinates by permutations and multiplications by unimodular numbers, then $\operatorname{Aut}(\mathbf C)$ acts by: pulling back $Q(R)$, stabilizing the operator $\caL$ and on the cotangents $-t_k\tilde\beta_k/\pi$, $t_k=z_kw_k$, $1\le k\le m$ by permutations and multiplications by unimodular numbers.
\end{corollary}
\begin{proof}
We showed in Theorem \ref{main} that an analytic equivalence of families acts by pulling back the vanishing residue sheaf.  In particular maps $h_X$ act by pulling back $\Pi_*h_X^*\caV$ the direct image of the vanishing residue sheaf.  In particular for $\mathbf a\in\operatorname{Aut}(\mathbf C)$ and associated map $h_X$, we have by definition of the extending sheaf sections map that $h_X^*\caL_{\{z_k,w_k\}} (\beta)=\caL_{\{z_k\circ h_X,w_k\circ h_X\}} (h_X^*\beta)= \caL_{\{z_k\circ h_X,w_k\circ h_X\}} (\mathbf a^*\beta)$ for $\beta\in Q(R)$.  In general the  coordinates $\{z_k\circ h_X,w_k\circ h_X\}$ and $\{z_k,w_k\}$ are distinct and extensions of sections are not preserved by pulling back.  

The group $\operatorname{Aut}(\mathbf C)$ acting on the normalization $R$ has cyclic stabilizer at each distinguished point.  Local coordinates can be chosen with each stabilizer acting by rotations.  For such coordinates, $\operatorname{Aut}(\mathbf C)$ acts by permutations and rotations of coordinates.  For such coordinates, $\operatorname{Aut}(\mathbf C)$ acts preserving extensions of sections of $Q(R)$ to sections of $\Pi_*\caV$.  Further for such coordinates, $\operatorname{Aut}(\mathbf C)$ acts by a permutation of the sections $\tilde\beta_h$, $1\le h \le m$ and by the same permutation on the coordinate indices and by multiplication by unimodular numbers on the coordinate products. 
\end{proof}

\subsection{Plumbing elliptic curves and a pointed $\mathbb P^1$.}

We describe plumbing elliptic curves (flat tori) to a projective line $\mathbb P^1$  pointed at $0$, $-1$ and $1$.  The general plumbing gives a genus $2$ surface with hyperelliptic involution.  We consider the family cotangent frame and the action of the automorphism group of the initial fiber.  The construction begins with an elliptic curve $\mathcal E\subtau$ with variable $z$, universal cover $\mathbb C$ and lattice deck transformation group generated by $1$ and $\tau\in\mathbb H$.  An elliptic curve has quadratic differentials $dz^2$ and $\bP=\caP dz^2$ for the Weierstrass $\caP$-function.  The $\caP$-function is doubly periodic with singular part $1/z^2$ at the origin; $\bP$ has unit residue at the origin.  An elliptic curve has the involution $\imath\subE: z\rightarrow -z$ with the differentials $dz^2$ and $\bP$ involution invariant.  The projective line $\mathbb P^1$ with variable $w$ has Abelian differentials $\omega_{ab}=(a-b)dw/((w-a)(w-b))$ with residues $1$ at $a$ and $-1$ at $b$.  The projective line 
$\mathbb P^1$ pointed at $0$, $-1$ and $1$ has the involution $\imath\subP:w\rightarrow -w$. 

Define a nodal stable curve $\mathbf C$ by pairing distinguished points as follows.  Pair the origins on $\caE\subtau$ and $\mathbb P^1$ to form a first node $\bn\subE$, and pair the points $-1$ and $1$ on $\mathbb P^1$ to form a second node $\bn\subP$.  The involutions $\imath\subE$ and $\imath\subP$ of components define involutions on the stable curve by extending by the identity on the remaining component.  Define plumbings of the stable curve $\mathbf C$ as follows.  At the node $\bn\subE$, use the local coordinates 
$u\subE=z$, $v\subE=w$ and write $u\subE v\subE=t\subE$ for plumbing.  At the node $\bn\subP$, use the local coordinates $u\subP=w+1$, $v\subP=w-1$ and write $u\subP v\subP=t\subP$ for plumbing.  The elliptic curves $\caE\subtau$ and plumbings define a family $\caR$ with general fiber a genus $2$ surface over a $3$-dimensional base $\caU$ with base parameters $(t\subE,t\subP,\tau)$.  Following Corollary \ref{aut}, the action of the $\mathbf C$ involutions on the local coordinates at distinguished points is as follows.  The involution $\imath\subE$ maps $u\subE$ to $-u\subE$ and fixes the remaining coordinates.  The involution $\imath\subP$ fixes $u\subE$, maps $v\subE$ to $-v\subE$, $u\subP$ to $-v\subP$ and $v\subP$ to $-u\subP$. It follows immediately that the involutions each act on the $\caU$ parameters by the map $(t\subE,t\subP,\tau)\rightarrow (-t\subE,t\subP,\tau)$. 

We describe sections of the $\mathbf C$ dualizing sheaf square and following the discussion after Lemma \ref{cotfib}, describe the extension of the sections by the linear map $\Lsubuv$  to sections of the relative dualizing sheaf of $\caR$ over $\caU$.  Given the action of the involutions on the $u$, $v$ local coordinates, by Corollary \ref{aut} the involutions $\imath\subE$ and $\imath\subP$ stabilize the linear map $\Lsubuv$.  To define sections of the $\mathbf C$ dualizing sheaf square, begin with the $\caE$ node and consider $(-1/4\pi)(\bP+\omega_{0-1}\omega_{01})$.  By construction, the section has residue $-1/4\pi$ at $\bn\subE$ and residue $0$ at $\bn\subP$ with simple poles at the points $\pm 1$.  The section is invariant by the involutions $\imath\subE$ and $\imath\subP$.  We write $\beta\subE$ for the extension $\Lsubuv((-1/4\pi)(\bP+\omega_{0-1}\omega_{01}))$ to a section of the relative dualizing sheaf $\Pi_*\omega^2_{\caR/\caU}$ with $\Laur_{\bn\subE}(\beta\subE)=-1/\pi$ and $\Laur_{\bn\subP}(\beta\subE)=0$.  The section $\beta\subE$ is invariant by the involutions acting on $\caR$, since the dualizing sheaf section is invariant and the involutions stabilize the map $\Lsubuv$.  
Next consider the $\mathbb P^1$ node and the section $(-1/4\pi)(\omega_{-11}^2)$ extended by zero on $\caE$.  By construction, the section has residue $0$ at $\bn\subE$ and $-1/4\pi$ at $\bn\subP$.  The section is invariant by the involutions.    
We write $\beta\subP$ for the extension $\Lsubuv((-1/4\pi)(\omega_{-11}^2))$ to a section of the relative dualizing sheaf $\Pi_*\omega^2_{\caR/\caU}$ with 
$\Laur_{\bn\subE}(\beta\subP)=0$ and $\Laur_{\bn\subP}(\beta\subP)=-1/\pi$.  Similar to $\beta\subE$, the section $\beta\subP$ is invariant by the involutions acting on $\caR$.    Finally consider the section $-2idz^2$ on $\caE$, extended by zero on $\mathbb P^1$.  The section is invariant by the involutions. We write $\beta\subdz$ for the extension $\Lsubuv(-2idz^2)$ to a section of the relative dualizing sheaf $\Pi_*\omega^2_{\caR/\caU}$ with vanishing $\Laur$ coefficients.  Similar to $\beta\subE$ and $\beta\subP$, the section $\beta\subdz$ is invariant by the involutions acting on $\caR$. 

We consider the expansions of the sections $t\subE\beta\subE,t\subP\beta\subP,\beta\subdz$ at nodes as examples of our overall considerations.  For a node $uv=t$ and a section $\eta$ of $\omega_{\caR/\caU}^2$ given as
\[
\bff(u,v,s)\big(\frac{du}{u}-\frac{dv}{v}\big)^2,
\]
$s$ representing the remaining variables, the interpretation of initial derivatives of $\bff$ is as follows.  The vanishing residue condition is $\bff(0,0,s)=0$. The derivatives $\bff_u(0,0,s)$ and $\bff_v(0,0,s)$ are the coefficients of the terms $(du)^2/u$ and $(dv)^2/v$ in the expansion of $\eta$ on the branches of the normalization of the node.  Finally the derivative $4\bff_{uv}(0,0,s)$ is the $t$-linear term of the function $\Laur(\eta)$ of $t$.  At the node $\bn\subE$, the local coordinates for $\caR$ are $u\subE,v\subE,t\subP$ and $\tau$.  At the node $\bn\subP$, the local coordinates for $\caR$ are  
$u\subP,v\subP,t\subE$ and $\tau$. For the sections $t\subE\beta\subE,t\subP\beta\subP,\beta\subdz$, the expansions $\bff_uu+\bff_vv+\bff_{uv}uv$ modulo $O(u^2)+O(v^2)$ remainders are as follows. 
The section $t\subE\beta\subE$ has expansion $-u\subE v\subE/4\pi$ at $\bn\subE$ and $0$ at $\bn\subP$.  The section $t\subP\beta\subP$ has expansion $-u\subP v\subP/4\pi$ at $\bn\subP$ and $0$ at $\bn\subE$.  The section $\beta\subdz$ has $0$ expansion at each node.  It now follows from Lemma \ref{plumderiv} that on the domain of sections, the pairing of plumbing tangents with the sections $t\subE\beta\subE,t\subP\beta\subP,\beta\subdz$ is given as follows: $\partial/\partial t\subE$ pairs to the values $(1,0,0)$ and $\partial/\partial t\subP$ pairs to the values $(0,1,0)$.  Furthermore by the Gr\"{o}tzsch and Rauch variational formula, for $t\subE=t\subP=0$ the section $\beta\subdz$ represents the cotangent $d\tau$ and thus $\partial/\partial\tau$ pairs with the sections $t\subE\beta\subE,t\subP\beta\subP,\beta\subdz$ to the values $(0,0,1)$.  
For $t\subE=t\subP=0$, the sections $t\subE\beta\subE,t\subP\beta\subP,\beta\subdz$ are the coordinate cotangent frame for the base $\caU$.  
Combining the $\imath\subE$ and $\imath\subP$ invariance of the sections $\beta\subE,\beta\subP,\beta\subdz$ and the involution action on coordinates $(t\subE,t\subP,\tau)$ gives the expected action $(t\subE\beta\subE,t\subP\beta\subP,\beta\subdz)\rightarrow
(-t\subE\beta\subE,t\subP\beta\subP,\beta\subdz)$ on cotangent representatives. 
     
By Theorem \ref{DM}, the $(s,t)$ family $\caR$ over $\caU$ with $\mathbf C$ automorphism action provides a local analytic description for the Deligne-Mumford compactification. The action of the involutions follows general expectations, \cite[Chapter 11, Proposition 4.11]{ACG2}.  The product $\imath\subE\imath\subP$ acts on the fiber $\mathbf C$ as the limit of hyperelliptic involutions and for genus $2$ the action is everywhere trivial.  The action of a single involution is the standard {\em half Dehn twist for a $1$-handle}, or equivalently the {\em $2$-torsion associated with an elliptic tail.}   

\subsection{An example of plumbing an Abelian differential.}

We apply a standard construction for plumbing an Abelian differential and then calculate the variation of its period.  The plumbing provides an example of a section of the direct image of $\caV$ as a cotangent sheaf section and a demonstration for Rauch's period variation formula \cite{Rauch}.  We calculate the period variation by two approaches and then compare.  

Begin with a compact Riemann surface $R$ with a canonical homology basis $\{A_j,B_j\}_{j=1}^g$, given by representative cycles and with points $a,b$ disjoint from the cycles. By 
Riemann-Roch there is a meromorphic Abelian differential with residue $-1$ at $b$ and $1$ at $a$ \cite{GunRS}.  The meromorphic differential is unique modulo the analytic differentials. A unique differential $\omega$ of the third kind is determined by the condition of vanishing periods on the given homology basis.  Introduce local coordinates $u$ at $a$ and $v$ at $b$, such that $u(a)=v(b)=0$, and such that the charts include the discs $\{|u|\le1\},\,\{|v|\le1\}$. Coordinates $u,v$ are analytically specified by the condition that on the discs, $\omega$ is respectively given as $du/2\pi iu$ and as $-dv/2\pi iv$.   We can use the coordinates to describe points near $a$ or $b$.  Introduce an arc $\hat\gamma$ from $v=1$ to $u=1$, disjoint from the basis of cycles and contained in $R-\{|u|<1\}-\{|v|<1\}$. 

The plumbing of $R$ is as follows.  Given $t$ nonzero, remove the closed discs $\{|v|\le|t|\}$ and $\{|u|\le|t|\}$.  Overlap the annuli $\{|t|<|u|\}$ and $\{|t|<|v|<1\}$ by $uv=t$ to obtain the plumbed surface $R_t$.  Since the identification is $u=t/v$, the differential $\omega$ plumbs to a differential $\omega_t$ on $R_t$.  Introduce an oriented arc $\gamma_t$ from $1$ to $t$ in the annulus $\{|t|\le|u|\le 1\}$ and disjoint from $\arg u=\pi$.  The arc concatenation $\gamma_t+\hat\gamma$ determines an oriented cycle on $R_t$, disjoint from the original basis of cycles.  The image of the positively oriented circle $\{|u|=1\}$ also defines a cycle $A_*$.  By construction the integral of $\omega_t$ over $A_*$ is unity and the genus of $R_t$ is one greater than the genus of $R$.  The combination of images of cycles $\{A_j,A_*,B_j,\gamma_t+\hat\gamma\}_{j=1}^g$ is a canonical homology basis for $R_t$.  The Abelian differential $\omega_t$ is the dual to $A_*$ relative to the $A$ cycles of the basis.  
The period
\[
\caP\,=\,\int_{\gamma_t+\hat\gamma}\omega_t
\]
is an entry in the Riemann period matrix of $R_t$.  We set $\hat\caP=\int_{\hat\gamma}\omega_t$ and from the above construction $\int_{\gamma_t}\omega_t=(\log t)/2\pi i$.  The exponential of the period 
\[
\exp(2\pi i\caP)=\exp(2\pi i\hat\caP)t
\]
extends to an analytic function on the disc $\{|t|<1\}$.  The derivative of the function is $\exp(2\pi i\hat\caP)$.  We can also calculate the derivative by combining Rauch's variational formula \cite[Theorem 2]{Rauch} and Lemma \ref{plumderiv}.  Preliminary observations are necessary.   First, by construction in a neighborhood of the plumbing, the $t$ family $R_t$ coincides with the family $\piVD$.  Second, $R_t$ locally embeds into a fiber of $\piVD$ by the 
map $u\rightarrow(u,t/u)$ and consequently $du/u$ is the pullback of $1/2(dz/z-dw/w)$ on $V$.   Third, Rauch's formula is given in terms of the $2$-form $dz\wedge d\bar z=-2idE$ for the Beltrami - quadratic differential pairing, whereas the present formulas are given in terms of the $2$-form $dE$; see (\ref{Serrepair}).  Applying the period variation formula, Lemma \ref{plumderiv} and the observations gives 
\[
\frac{d}{dt}\,e^{2\pi i\caP}\,=\,e^{2\pi i\caP}\,2\pi i\frac{d\caP}{dt}\,=\,e^{2\pi i\caP}\,2\pi i\,\frac{-\pi}{t}\Laur(-2i\omega_t^2).
\]
In a neighborhood of the plumbing, $\omega_t$ is given as 
\[
\omega_t= \frac{1}{4\pi i}\dzdw 
\] 
and 
\[
\Laur\big(\frac14\dzdw^2\big)=1.
\]  
Combining contributions, the derivative of the exponential of the period is $\exp(2\pi i\caP)/t=\exp(2\pi i\hat\caP)$, matching the direct calculation.       

The derivative formula can be expressed as the differential of a function by $d\exp(2\pi i\caP)=\exp(2\pi i\caP)4\pi\omega_t^2$.  The family $R_t$ is one dimensional with coordinate $t=\exp(2\pi i(\caP-\hat\caP))$ with $\hat\caP$ a constant. The derivative formula can also be given as $dt=4\pi t\omega_t^2$, with the right hand side a second frame element, in particular a section of the direct image sheaf $\Pi_*\caV$. 

%
%

\begin{thebibliography}{ACGH85}

\bibitem[AB60]{AB}
Lars Ahlfors and Lipman Bers.
\newblock Riemann's mapping theorem for variable metrics.
\newblock {\em Ann. of Math. (2)}, 72:385--404, 1960.

\bibitem[ACG11]{ACG2}
Enrico Arbarello, Maurizio Cornalba, and Phillip~A. Griffiths.
\newblock {\em Geometry of algebraic curves. {V}olume {II}}, volume 268 of {\em
  Grundlehren der Mathematischen Wissenschaften [Fundamental Principles of
  Mathematical Sciences]}.
\newblock Springer, Heidelberg, 2011.
\newblock With a contribution by Joseph Daniel Harris.

\bibitem[ACGH85]{ACGH}
E.~Arbarello, M.~Cornalba, P.~A. Griffiths, and J.~Harris.
\newblock {\em Geometry of algebraic curves. {V}ol. {I}}, volume 267 of {\em
  Grundlehren der Mathematischen Wissenschaften [Fundamental Principles of
  Mathematical Sciences]}.
\newblock Springer-Verlag, New York, 1985.

\bibitem[Ahl61]{Ahsome}
Lars~V. Ahlfors.
\newblock Some remarks on {T}eichm\"uller's space of {R}iemann surfaces.
\newblock {\em Ann. of Math. (2)}, 74:171--191, 1961.

\bibitem[Bar89]{Bard}
Fabio Bardelli.
\newblock Lectures on stable curves.
\newblock In {\em Lectures on {R}iemann surfaces ({T}rieste, 1987)}, pages
  648--704. World Sci. Publ., Teaneck, NJ, 1989.

\bibitem[Ber74]{Bersdeg}
Lipman Bers.
\newblock Spaces of degenerating {R}iemann surfaces.
\newblock In {\em Discontinuous groups and Riemann surfaces (Proc. Conf., Univ.
  Maryland, College Park, Md., 1973)}, pages 43--55. Ann. of Math. Studies, No.
  79. Princeton Univ. Press, Princeton, N.J., 1974.

\bibitem[Gar75]{Gardtheta}
Frederick~P. Gardiner.
\newblock Schiffer's interior variation and quasiconformal mapping.
\newblock {\em Duke Math. J.}, 42:371--380, 1975.

\bibitem[Gra60]{Graucoh}
Hans Grauert.
\newblock Ein {T}heorem der analytischen {G}arbentheorie und die {M}odulr\"aume
  komplexer {S}trukturen.
\newblock {\em Inst. Hautes \'Etudes Sci. Publ. Math.}, (5):64, 1960.

\bibitem[Gun66]{GunRS}
R.~C. Gunning.
\newblock {\em Lectures on {R}iemann surfaces}.
\newblock Princeton Mathematical Notes. Princeton University Press, Princeton,
  N.J., 1966.

\bibitem[Hej78]{Hejmono}
Dennis~A. Hejhal.
\newblock Monodromy groups and {P}oincar\'e series.
\newblock {\em Bull. Amer. Math. Soc.}, 84(3):339--376, 1978.

\bibitem[HK11a]{HKmg}
John~H. Hubbard and Sarah Koch.
\newblock An analytic construction of the {D}eligne-{M}umford compactification
  of the moduli space of curves.
\newblock preprint, 2011.

\bibitem[HK11b]{HKlet}
John~H. Hubbard and Sarah Koch.
\newblock Letters.
\newblock 2011.

\bibitem[HM98]{HMbook}
Joe Harris and Ian Morrison.
\newblock {\em Moduli of curves}, volume 187 of {\em Graduate Texts in
  Mathematics}.
\newblock Springer-Verlag, New York, 1998.

\bibitem[HSS09]{HSS}
John Hubbard, Dierk Schleicher, and Mitsuhiro Shishikura.
\newblock Exponential {T}hurston maps and limits of quadratic differentials.
\newblock {\em J. Amer. Math. Soc.}, 22(1):77--117, 2009.

\bibitem[Kod86]{Kod}
Kunihiko Kodaira.
\newblock {\em Complex manifolds and deformation of complex structures}.
\newblock Springer-Verlag, New York, 1986.
\newblock Translated from the Japanese by Kazuo Akao, With an appendix by
  Daisuke Fujiwara.

\bibitem[Kra85]{Kracusp}
Irwin Kra.
\newblock Cusp forms associated to loxodromic elements of {K}leinian groups.
\newblock {\em Duke Math. J.}, 52(3):587--625, 1985.

\bibitem[LSY04]{LSY1}
Kefeng Liu, Xiaofeng Sun, and Shing-Tung Yau.
\newblock Canonical metrics on the moduli space of {R}iemann surfaces. {I}.
\newblock {\em J. Differential Geom.}, 68(3):571--637, 2004.

\bibitem[LSY05]{LSY2}
Kefeng Liu, Xiaofeng Sun, and Shing-Tung Yau.
\newblock Canonical metrics on the moduli space of {R}iemann surfaces. {II}.
\newblock {\em J. Differential Geom.}, 69(1):163--216, 2005.

\bibitem[Mas76]{Msext}
Howard Masur.
\newblock Extension of the {W}eil-{P}etersson metric to the boundary of
  {T}eichmuller space.
\newblock {\em Duke Math. J.}, 43(3):623--635, 1976.

\bibitem[McM00]{McM}
Curtis~T. McMullen.
\newblock The moduli space of {R}iemann surfaces is {K}\"ahler hyperbolic.
\newblock {\em Ann. of Math. (2)}, 151(1):327--357, 2000.

\bibitem[Nar95]{Narbk}
Raghavan Narasimhan.
\newblock {\em Several complex variables}.
\newblock Chicago Lectures in Mathematics. University of Chicago Press,
  Chicago, IL, 1995.
\newblock Reprint of the 1971 original.

\bibitem[Pet39]{HPZur}
Hans Petersson.
\newblock Zur analytischen {T}heorie der {G}renzkreisgruppen.
\newblock {\em Math. Z.}, 44(1):127--155, 1939.

\bibitem[Pet41]{HPEin}
Hans Petersson.
\newblock Einheitliche {B}egr\"undung der {V}ollst\"andigkeits\-s\"atze f\"ur
  die {P}oincar\'eschen {R}eihen von reeller {D}imension bei beliebigen
  {G}renzkreisgruppen von erster {A}rt.
\newblock {\em Abh. Math. Sem. Hansischen Univ.}, 14:22--60, 1941.

\bibitem[Rau59]{Rauch}
H.~E. Rauch.
\newblock Weierstrass points, branch points, and moduli of {R}iemann surfaces.
\newblock {\em Comm. Pure Appl. Math.}, 12:543--560, 1959.

\bibitem[Wel08]{Wells}
Raymond~O. Wells, Jr.
\newblock {\em Differential analysis on complex manifolds}, volume~65 of {\em
  Graduate Texts in Mathematics}.
\newblock Springer, New York, third edition, 2008.
\newblock With a new appendix by Oscar Garcia-Prada.

\bibitem[\Wp88]{Wlcut}
Scott~A. \WpName{Wolpert}.
\newblock Cut-and-paste deformations of {R}iemann surfaces.
\newblock {\em Ann. Acad. Sci. Fenn. Ser. A I Math.}, 13(3):401--413, 1988.

\bibitem[\Wp90]{Wlhyp}
Scott~A. \WpName{Wolpert}.
\newblock The hyperbolic metric and the geometry of the universal curve.
\newblock {\em J. Differential Geom.}, 31(2):417--472, 1990.

\bibitem[\Wp03]{Wlcomp}
Scott~A. \WpName{Wolpert}.
\newblock {G}eometry of the {W}eil-{P}etersson completion of {T}eichm\"{u}ller
  space.
\newblock In {\em Surveys in Differential Geometry VIII: Papers in Honor of
  Calabi, Lawson, Siu and Uhlenbeck}, pages 357--393. Intl. Press, Cambridge,
  MA, 2003.

\bibitem[\Wp10a]{Wlcbms}
Scott~A. \WpName{Wolpert}.
\newblock {\em Families of {R}iemann surfaces and {W}eil-{P}etersson
  {G}eometry}, volume 113 of {\em CBMS Regional Conference Series in
  Mathematics}.
\newblock Published for the Conference Board of the Mathematical Sciences,
  Washington, DC, 2010.

\bibitem[\Wp10b]{Wlcurv}
Scott~A. \WpName{Wolpert}.
\newblock Geodesic-length functions and the {W}eil-{P}etersson curvature
  tensor.
\newblock J. Differential Geom., to appear, Arxiv.org/1008.2293, 2010.

\end{thebibliography}

\providecommand\WlName[1]{#1}\providecommand\WpName[1]{#1}\providecommand\Wl{W%
lf}\providecommand\Wp{Wlp}\def\cprime{$'$}

\end{document}